\documentclass[12pt,leqno]{article}
\usepackage{graphicx}
\usepackage{amsthm}
\usepackage{hyperref}
\usepackage{amssymb,amsmath,amsthm}
\usepackage{epstopdf}
\usepackage{mathrsfs}
\usepackage{xypic}
\usepackage{comment}
\usepackage{mathabx,epsfig}
\usepackage{color}
\theoremstyle{plain}
\newtheorem{theorem}{Theorem}[section]
\newtheorem{observation}[theorem]{Observation}
\newtheorem{corollary}[theorem]{Corollary}

\newtheorem{proposition}[theorem]{Proposition}
\newtheorem{lemma}[theorem]{Lemma}

\theoremstyle{definition}

\newtheorem{remark}[theorem]{Remark}

\usepackage{comment}
\usepackage{cancel}

\def\NN{\mathbb{N}}

\def\QQ{\mathbb{Q}}
\def\ZZ{\mathbb{Z}}

\def\pet{\operatorname{pet}}

\newcommand{\sqbinom}[2]{\genfrac{[}{]}{0pt}{}{#1}{#2}}

\def\a{\mathbf{a}}

\def\x{\mathbf{x}}

\def\0{\mathbf{0}}
\def\bomega{\boldsymbol{\omega}}

\def\downstrut{\noindent\lower5pt\vbox{}}

\allowdisplaybreaks

\title{Cyclic Sieving of Multisets with Bounded Multiplicity and the Frobenius Coin Problem}
\author{Drew Armstrong}
\date{}

\begin{document}

\maketitle

\begin{abstract}
The two subjects in the title are related via the specialization of symmetric polynomials at roots of unity. Let $f(z_1,\ldots,z_n)\in\ZZ[z_1,\ldots,z_n]$ be a symmetric polynomial with integer coefficients and let $\omega$ be a primitive $d$th root of unity. If $d|n$ or $d|(n-1)$ then we have $f(1,\omega,\ldots,\omega^{n-1})\in\ZZ$. If $d|n$ then of course we have $f(\omega,\omega^2,\ldots,\omega^n)=f(1,\omega,\ldots,\omega^{n-1})\in\ZZ$, but when $d|(n+1)$ we also have $f(\omega,\omega^2,\ldots,\omega^n)\in\ZZ$. We investigate these three families of integers in the case $f=h_k^{(b)}$, where $h_k^{(b)}$ is the coefficient of $t^k$ in the generating function $\prod_{i=1}^n (1+z_it+\cdots+(z_it)^{b-1})$. These polynomials were previously considered by several authors. They interpolate between the elementary symmetric polynomials ($b=2$) and the complete homogeneous symmetric polynomials ($b\to\infty$). When $\gcd(b,d)=1$ with $d|n$ or $d|(n-1)$ we find that the integers $h_k^{(b)}(1,\omega,\ldots,\omega^{n-1})$ are related to cyclic sieving of multisets with multiplicities bounded above by $b$, generalizing the well-known cyclic sieving results for sets ($b=2$) and multisets ($b\to \infty$). When $\gcd(b,d)=1$ and $d|(n+1)$ we find that the integers $h_k^{(b)}(\omega,\omega^2,\ldots,\omega^n)$ are related to the Frobenius coin problem with two coins. The case $\gcd(b,d)\neq 1$ is more complicated. At the end of the paper we combine these results with the expansion of $h_k^{(b)}$ in various bases of the ring of symmetric polynomials.
\end{abstract}

\tableofcontents

\section{Summary}\label{sec:summary}

We begin with a brief summary of our main results. Let the numbers $\binom{n}{k}^{(b)}\in\NN$ be defined by
\begin{equation*}
\sum_k \binom{n}{k}^{(b)} t^k =(1+t+\cdots+t^{b-1})^n
\end{equation*}
and let the polynomials $\sqbinom{n}{k}_q^{(b)}\in\NN[q]$ be defined by
\begin{equation*}
\sum_k \sqbinom{n}{k}^{(b)}_q t^k = \prod_{i=0}^{n-1} (1+q^i t+\cdots+(q^it)^{b-1}).
\end{equation*}
If $n$ or $k$ is not an integer we define $\binom{n}{k}^{(b)}:=0$ and $\sqbinom{n}{k}_q^{(b)}:=0$.

Let $\omega$ be a primitive $d$th root of unity. It is a basic fact (see Observation \ref{obs:galois} and Remark \ref{rem:elementary}) that $\sqbinom{n}{k}_\omega^{(b)}$ is an integer when $d|n$ or $d|(n-1)$, and that $\omega^k \sqbinom{n}{k}_\omega^{(b)}$ is an integer when $d|(n+1)$. We are interested in combinatorial interpretations of these integers.

If $d|n$ and $\gcd(b,d)=1$ then we have (Theorem \ref{thm:main}a):
\begin{equation*}
\sum_k \sqbinom{n}{k}^{(b)}_\omega t^k = (1+t^d+\cdots+(t^d)^{b-1})^{n/d},
\end{equation*}
which implies that
\begin{equation*}
\sqbinom{n}{k}_\omega^{(b)}=\binom{n/d}{k/d}^{(b)}.
\end{equation*}
Let $X$ be the set of $k$-multisubsets from $\{1,\ldots,n\}$ whose multiplicities are less than $b$. It is easy to check (see Theorem \ref{thm:csp}a) that $\binom{n/d}{k/d}^{(b)}$ is also equal to the number of elements of $X$ fixed by the permutation $\rho$ of $\{1,\ldots,n\}$ defined by $i\mapsto i+n/d$ mod $n$. The situation for $\gcd(b,d)\neq 1$ is more complicated.

If $d|(n-1)$ and $\gcd(b,d)=1$ then we have (Theorem \ref{thm:main}b):
\begin{equation*}
\sum_k  \sqbinom{n}{k}^{(b)}_\omega t^k = (1+t+\cdots+t^{b-1})(1+t^d+\cdots+(t^d)^{b-1})^{(n-1)/d},
\end{equation*}
which implies that
\begin{equation*}
\sqbinom{n}{k}_\omega^{(b)} = \sum_{\ell=0}^{b-1} \binom{(n-1)/d}{(k-\ell)/d}^{(b)}.
\end{equation*}
It is easy to check (see Theorem \ref{thm:csp}b) that this number also counts the elements of $X$ fixed by the permutation $\tau:\{1,\ldots,n\}$ defined by $n\mapsto n$ and $i\mapsto i+(n-1)/d$ mod $n-1$ for $i\in\{1,\ldots,n-1\}$. The situation for $\gcd(b,d)\neq 1$ is more complicated. These results generalize the two most basic examples of cyclic sieving \cite[Theorem 1.1]{rsw}, which correspond to the cases $b=2$ (subsets) and $b\to\infty$ (unbounded multisets).

If $d|(n+1)$ and $\gcd(b,d)=1$ then we have (Theorem \ref{thm:main}c):
\begin{equation*}
\sum_k \omega^k \sqbinom{n}{k}^{(b)}_\omega t^k = \frac{(1+t^d+\cdots+(t^d)^{b-1})^{(n+1)/d}}{1+t+\cdots+t^{b-1}},
\end{equation*}
which is a polynomial in $\ZZ[t]$. This time we do not have an interpretation in terms of cyclic actions on multisets. Instead, we find that the coefficients are related to the Frobenius coin problem. As a teaser, we mention one interesting result (Corollary \ref{cor:interesting}). Let $(\delta_0,\ldots,\delta_{d-1})$ be the permutation of $(0,\ldots,d-1)$ defined by $\delta_r b\equiv r$ mod $d$. Then for all $r\in\{0,\ldots,d-1\}$ we have
\begin{equation*}
\sum_{k\geq 0} \omega^{r+kd}\sqbinom{n}{r+kd}^{(b)}_\omega t^{r+kd} = \varepsilon_r f_r(t^d),
\end{equation*}
where $f_r(t)\in\NN[t]$ is a polynomial with positive, unimodal coefficients, and where
\begin{equation*}
\varepsilon_r = \begin{cases} +1 & \delta_r<\delta_1, \\ -1 & \delta_r\geq \delta_1.\end{cases}
\end{equation*}
The situation for $\gcd(b,d)\neq 1$ is more complicated.

At the very end of this project we stumbled upon a surprising formula. Recall that a partition $\lambda$ of length $\ell=l(\lambda)$ is a sequence $\lambda=(\lambda_1,\ldots,\lambda_\ell)\in\NN^\ell$ where $\lambda_1\geq \cdots\geq \lambda_\ell\geq 1$. We write $|\lambda|=k$ when $\sum_i \lambda_i=k$. Let $m_i=\#\{j:\lambda_j=i\}$ be the number of parts of $\lambda$ equal to $i$ and let $z_\lambda=\prod_i i^{m_i} \cdot m_i!$. For each $b\geq 2$ let $l_b(\lambda)=\#\{j: b|\lambda_j \}$ be the number of parts of $\lambda$ divisible by $b$. Then for all $n,k,b\in\NN$ we have (Proposition \ref{prop:hbpow})
\begin{equation*}
\binom{n}{k}^{(b)} = \sum_{|\lambda|= k} z_\lambda^{-1} (1-b)^{l_b(\lambda)} n^{l(\lambda)}.
\end{equation*} 
We obtain this as the specialization of an identity involving symmetric polynomials. Corollary \ref{cor:hbpow} gives the analogous specializations at roots of unity.

The paper is organized as follows. Section \ref{sec:csp} discusses the cyclic sieving phenomenon. In this section we state and partially prove Theorems \ref{thm:csp}a and \ref{thm:csp}b on cyclic sieving of multisets with bounded multiplicity. In Section \ref{sec:specialization} we discuss the $b$-bounded symmetric polynomials and their specializations at roots of unity. The main result is Theorem \ref{thm:main}. We use parts (a) and (b) of Theorem \ref{thm:main} to complete the proof of Theorem \ref{thm:csp}. Part (c) of Theorem \ref{thm:main} is not related to cyclic sieving. Instead, we show that it is related to the Frobenius coin problem, which we discuss in Section \ref{sec:sylvester}. Finally, in Section \ref{sec:other} we discuss the expansion of $b$-bounded symmetric polynomials in various bases for the ring of symmetric polynomials. In particular, the expansion in power sums yields the surprising formula.

\section{Cyclic Sieving}\label{sec:csp}

This investigation began with the experimental observation (see Theorem \ref{thm:csp}a) that certain collections of multisets with bounded multiplicity obey a cyclic sieving property. We recall the setup from Reiner, Stanton and White \cite{rsw}.\footnote{See \cite{sagan} for a more recent survey.} Suppose that a cyclic group $\langle \rho\rangle$ acts on a finite set $X$. Let $f(q)=\sum_{x\in X} q^{\text{stat}(x)}$ be the generating polynomial for a certain function $\text{stat}:X\to \NN$, so $f(1)=\#X$. Let $N=\#\langle\rho\rangle$. We say that the triple $(X, f(q), \langle \rho\rangle)$ exhibits the {\em cyclic sieving phenomenon} (CSP) if for any primitive $d$th root of unity $\omega$ with $d|N$ we have
\begin{equation*}
f(\omega)=\#X^{\rho^{N/d}} :=\#\{x\in X:\rho^{N/d}\cdot x=x\}.
\end{equation*}
Reiner, Stanton and White observed that this concept unifies many disparate examples in which a combinatorially defined polynomial $f(q)$ yields integers when evaluated at certain roots of unity. The most basic examples of CSP occur when $f(q)$ is a $q$-binomial coefficient, which we define as usual:
\begin{equation*}
[n]_q = 1+q+\cdots+q^{n-1},\,\,  [n]_q! = [n]_q[n-1]_q\cdots [1]_q  \,\,\text{ and }\,\, \sqbinom{n}{k}_q = \frac{[n]_q!}{[k]_q![n-k]_q!}.
\end{equation*}
It is well-known that $\sqbinom{n}{k}_q$ is a polynomial with positive integer coefficients. A {\em $k$-multisubset} of $\{1,\ldots,n\}$ has the form
\begin{equation*}
\{1^{x_1},2^{x_2},\ldots,n^{x_n}\} = \{\underbrace{1,\ldots,1}_{\text{$x_1$ times}},\underbrace{2,\ldots,2}_{\text{$x_2$ times}}\ldots,\underbrace{n,\ldots,n}_{\text{$x_n$ times}}\},
\end{equation*}
where the symbol $i$ is repeated $x_i$ times and where $x_1+\cdots+x_n=k$. We can identify each multiset with its vector of multiplicities $\x=(x_1,\ldots,x_n)\in\NN^n$. Subsets without repetition correspond to binary vectors $(x_1,\ldots,x_n)\in\{0,1\}^n$. More generally, for any integers $b,n,k\in\NN$ we consider $k$-multisubsets with multiplicities bounded above by $b$:
\begin{align}
X_b^n &:=  \{(x_1,\ldots,x_n)\in\NN^n: x_i<b \text{ for all $i$}\}, \nonumber  \\
X_b^n[k] &:= \{(x_1,\ldots,x_n)\in\NN^n: x_i<b \text{ for all $i$ and } \sum_i x_i=k\}. \label{eq:xset}
\end{align}
The notation is suggested by the fact that $\#X_b^n = b^n$. Note that $X_b^n[k]=X_{k+1}^n[k]$ for all $b\geq k$ hence we may define $X_\infty^n[k]:=X_{k+1}^n[k]$. Thus $X_2^n[k]$ is the set of $k$-subsets and $X_\infty^n[k]$ is the set of $k$-multisubsets of $\{1,\ldots,n\}$. We have
\begin{align*}
\#X_2^n[k] &= \binom{n}{k},\\
\#X_\infty^n[k] &= \binom{n-1+k}{k}.
\end{align*}
Consider the permutations $\rho,\tau:X_b^n[k]\to X_b^n[k]$ where $\rho$ cycles all $n$ coordinates and $\tau$ cycles the first $n-1$ coordinates, leaving the $n$th coordinate fixed:
\begin{align*}
\rho\cdot (x_1,x_2,\ldots,x_n) &:= (x_2,\ldots,x_n,x_1),\\
\tau\cdot (x_1,x_2,\ldots,x_n) &:= (x_2,\ldots,x_{n-1},x_1,x_n).
\end{align*}
Here is the most basic result in the theory of cyclic sieving.

\begin{proposition}[{\cite[Theorem 1.1]{rsw}}]\label{prop:csp}
If $C=\langle \rho\rangle$ or $C=\langle\tau\rangle$ then the following exhibit CSP:
\begin{equation*}
\left( X_2^n[k], \sqbinom{n}{k}_q, C\right), \quad  \left( X_\infty^n[k], \sqbinom{n-1+k}{k}_q, C\right).
\end{equation*}
In other words, if $\omega$ is a primitive $d$th root of unity with $d|n$ then we have
\begin{align*}
\sqbinom{n}{k}_\omega &= \#X_2^n[k]^{\rho^{n/d}},\\
\sqbinom{n-1+k}{k}_\omega &=\#X_\infty^n[k]^{\rho^{n/d}}.
\end{align*}
and if $\omega$ is a primitive $d$th root of unity with $d|(n-1)$ then we have
\begin{align*}
\sqbinom{n}{k}_\omega &= \#X_2^n[k]^{\tau^{(n-1)/d}},\\
\sqbinom{n-1+k}{k}_\omega &=\#X_\infty^n[k]^{\tau^{(n-1)/d}}.
\end{align*}
\end{proposition}

In general the set of vectors $X_b^n[k]$ corresponds to $k$-multisubsets of $\{1,\ldots,n\}$ whose multiplicities are bounded above by $b$. We will use the notation
\begin{equation*}
\binom{n}{k}^{(b)}=\#X_b^n[k],
\end{equation*}
so that
\begin{equation*}
\binom{n}{k}^{(2)}=\binom{n}{k} \quad \text{ and } \quad \binom{n}{k}^{(\infty)}=\binom{n-1+k}{k}.
\end{equation*}
The generating function is
\begin{equation*}
\sum_{k=0}^{(b-1)n} \binom{n}{k}^{(b)} t^k = (1+t+\cdots+t^{b-1})^n.
\end{equation*}
These numbers have a long history and appear often in probability because $\binom{n}{k}^{(b)}/b^n$ is the chance of getting a sum of $k$ in $n$ rolls of a fair die with sides labeled $\{0,1,\ldots,b-1\}$. Belbachir and Igueroufa \cite{belbachir} have collected an extensive list of references going back to Euler. It is surprising that there is no standard notation; we use a modified form of Euler's notation.

Our main results in this paper have to do with the natural $q$-analogue defined by
\begin{equation*}
\sum_{k=0}^{(b-1)n} \sqbinom{n}{k}_q^{(b)} t^k = \prod_{i=0}^{n-1} (1+q^it+(q^it)^2+\cdots+(q^i t)^{b-1}) = \prod_{i=0}^{n-1} [b]_{q^i t},
\end{equation*}
which is the generating function for the statistic $x_2+2x_3+\cdots+(n-1)x_n$ on the set $X_b^n[k]$:
\begin{equation*}
\sqbinom{n}{k}_q^{(b)} = \sum_{\x\in X_b^n[k]} q^{x_2+2x_3+\cdots+(n-1)x_n}.
\end{equation*}
This is a generalization of the $q$-binomial coefficients in the sense that
\begin{align*}
\sqbinom{n}{k}_q^{(2)} &= q^{k(k-1)/2} \sqbinom{n}{k}_q,\\
\sqbinom{n}{k}_q^{(\infty)} &= \sqbinom{n-1+k}{k}_q.
\end{align*}
Here is the result that motivated the current paper.

\begin{theorem}[Cyclic Sieving of Multisets with Bounded Multiplicity]\label{thm:csp} Let $X_b^n[k]$ be the set of $n$-tuples $(x_1,\ldots,x_n)\in \{0,1,\ldots,b-1\}^n$ with $x_1+\cdots+x_n=k$ and consider the permutations $\rho,\tau$ of $X_b^n[k]$ defined by
\begin{align*}
\rho\cdot (x_1,x_2,\ldots,x_n) &:= (x_2,\ldots,x_n,x_1),\\
\tau\cdot (x_1,x_2,\ldots,x_n) &:= (x_2,\ldots,x_{n-1},x_1,x_n).
\end{align*}
\begin{enumerate}
\item If $\gcd(b,n)=1$ then the triple $(X_b^n[k], \sqbinom{n}{k}_q^{(b)},\langle \rho\rangle)$ exhibits the CSP. More precisely, if $\omega$ is a primitive $d$th root of unity with $d|n$ and $\gcd(b,d)=1$ then we have
\begin{equation*}
\sqbinom{n}{k}_\omega^{(b)} = \#X_b^n[k]^{\rho^{n/d}}.
\end{equation*}
\item If $\gcd(b,n-1)=1$ then the triple $(X_b^n[k], \sqbinom{n}{k}_q^{(b)},\langle \tau\rangle)$ exhibits the CSP. More precisely, if $\omega$ is a primitive $d$th root of unity with $d|(n-1)$ and $\gcd(b,d)=1$ then we have
\begin{equation*}
\sqbinom{n}{k}_\omega^{(b)} = \#X_b^n[k]^{\tau^{(n-1)/d}}.
\end{equation*}
\end{enumerate}
\end{theorem}

Here we will give the combinatorial portion of the proof. The algebraic portion of the proof follows from parts (a) and (b) of Theorem \ref{thm:main} in the next section.

\begin{proof}
Let $\omega$ be a primitive $d$th root of unity.

First suppose that $d|n$. If $d\nmid k$ then $X_b^n[k]^{\rho^{n/d}}=\emptyset$. If $d|k$ then we have a bijection $X_b^{n/d}[k/d]\to X_b^n[k]^{\rho^{n/d}}$ defined by 
\begin{equation*}
(x_1,\ldots,x_{n/d})\mapsto (x_1,\ldots,x_{n/d},\ldots,x_1,\ldots,x_{n/d}),
\end{equation*}
where the string $(x_1,\ldots,x_{n/d})$ is repeated $d$ times on the right hand side. Hence we have
\begin{equation*}
\#X_b^n[k]^{\rho^{n/d}}=\#X_b^{n/d}[k/d] = \binom{n/d}{k/d}^{(b)}.
\end{equation*}
In other words, we have the generating function
\begin{equation*}
\sum_k \#X_b^n[k]^{\rho^{n/d}}t^k=(1+t^d+\cdots+(t^d)^{b-1})^{n/d}.
\end{equation*}
Theorem \ref{thm:main}a below shows that the numbers $\sqbinom{n}{k}_\omega^{(b)}$ satisfy the same generating function, but only when $\gcd(b,d)=1$.

Next suppose that $d|(n-1)$. Any $\x\in X_b^n[k]^{\tau^{(n-1)/d}}$ has the form
\begin{equation*}
(x_1,\ldots,x_n) = (x_1,\ldots,x_{(n-1)/d},\ldots,x_1,\ldots,x_{(n-1)/d},x_n),
\end{equation*}
where the string $(x_1,\ldots,x_{(n-1)/d})$ is repeated $d$ times. Since the coordinates sum to $k$ we have $x_n+d \sum_{i=1}^{(n-1)/d} x_i=k$. If $x_n=\ell$ then $x_1,\ldots,x_{(n-1)/d}$ sum to $(k-\ell)/d$, hence there are $\binom{(n-1)/d}{(k-\ell)/d}^{(b)}$ ways to choose $x_1,\ldots,x_{(n-1)/d}$. It follows that
\begin{equation*}
\#X_b^n[k]^{\tau^{(n-1)/d}}=\sum_{\ell\geq 0} \binom{(n-1)/d}{(k-\ell)/d}^{(b)}.
\end{equation*}
In other words, we have the generating function
\begin{equation*}
\sum_k \#X_b^n[k]^{\tau^{(n-1)/d}}t^k=(1+t+\cdots+t^{b-1})(1+t^d+\cdots+(t^d)^{b-1})^{(n-1)/d}.
\end{equation*}
Theorem \ref{thm:main}b below shows that the numbers $\sqbinom{n}{k}_\omega^{(b)}$ satisfy the same generating function, but only when $\gcd(b,d)=1$.
\end{proof}

Let us observe how Theorem \ref{thm:csp} generalizes Proposition \ref{prop:csp}. Let $\omega$ be a primitive $d$th root of unity. If $b$ is larger than $k$ (and coprime to $d$) then we have
\begin{equation*}
\sqbinom{n}{k}^{(b)}_\omega =\sqbinom{n}{k}^{(\infty)}_\omega = \sqbinom{n-1+k}{k}_\omega,
\end{equation*}
so in this case we recover the multiset portion of Proposition \ref{prop:csp}. If $b=2$ then Theorem \ref{thm:csp} only applies when $d$ is odd, i.e., when $\gcd(2,d)=1$. In this case one can show with some tedious checking of signs and vanishing conditions that
\begin{equation*}
\sqbinom{n}{k}_\omega^{(2)}=\omega^{k(k-1)/2}\sqbinom{n}{k}_\omega=\sqbinom{n}{k}_\omega
\end{equation*}
whenever $d|n$ or $d|(n-1)$.\footnote{We regard this fact as a small coincidence. In general, it seems that one should just admit that $\sqbinom{n}{k}_q^{(2)}=q^{k(k-1)/2}\sqbinom{n}{k}_q$ is the appropriate $q$-analogue of binomial coefficients when dealing with sets, while $\sqbinom{n}{k}_q^{(\infty)}=\sqbinom{n-1+k}{k}_q$ is the appropriate analogue when dealing with multisets.}
If $d$ is even, so $\gcd(b,d)\neq 1$, it still turns out that $\sqbinom{n}{k}^{(2)}_\omega=\pm \sqbinom{n}{k}_\omega$, and the negative signs can be fixed by multiplying $\sqbinom{n}{k}_q^{(2)}$ with a suitable power of $q$. A similar phenomenon happens whenever $\gcd(b,d)=2$, but for general $\gcd(b,d)\neq 1$ with $d|n$ or $d|(n-1)$ the integers $\sqbinom{n}{k}_\omega^{(b)}$ seem unrelated to cyclic actions on $X_b^n[k]$. For example, consider $b=d=n=k=3$ with
\begin{equation*}
\sqbinom{3}{3}^{(b)}_q= q^4(1+2q+q^2+2q^3+q^4),
\end{equation*}
and consider the set
\begin{align*}
X_3^3[3] &= \{ (0,1,2), (0,2,1), (1,0,2), (1,2,0), (2,0,1), (2,1,0), (1,1,1)\}.
\end{align*}
Since $\rho$ acts on $X_3^3[3]$ by cycling all three coordinates we have $X_3^3[3]^\rho=\{(1,1,1)\}$, but if $\omega$ is a primitive $3$rd root of unity then
\begin{align*}
\sqbinom{3}{3}^{(3)}_\omega &= \omega^4(1+2\omega+\omega^2+2\omega^3+\omega^4)\\
&=\omega+2\omega^2+1+2\omega+\omega^2\\
&= 1+3\omega+3\omega^2\\
&= 1+3\omega+3(-1-\omega)\\
&= -2\\
&\neq \#X_3^3[3]^\rho.
\end{align*}

We first observed Theorem \ref{thm:csp}a in the case $b=n+1$ and then discovered the result for $\gcd(b,d)=1$ in the course of proving it. We discovered Theorem \ref{thm:csp}b much later, after a first draft of this paper was already written, and only then realized the connection to the ``nearly free cyclic action'' in Theorem 1.1 of Reiner, Stanton and White \cite{rsw}. We believe this accidental rediscovery confirms the inevitability of the idea.

Reiner, Stanton and White gave two proofs of Proposition \ref{prop:csp}. Their first proof interprets the $q$-binomial coefficients evaluated at roots of unity as character values for certain $GL_n$ representations, which can also be viewed as principal specializations of symmetric polynomials. Their second proof uses ``brute force'' arguments to show that enumeration and the $q$-substitutions have the same explicit formula.

Our point of view is that there is not a big distance between the two proofs. In Theorem \ref{thm:main} below we observe that Reiner, Stanton and White's ad hoc collection of $q$-binomial evaluations \cite[Proposition 4.2]{rsw} can be unified via specialization of certain symmetric polynomials at roots of unity. This is a case where generalization leads to simplification.

Unfortunately, it seems that a fully representation-theoretic proof of Theorem \ref{thm:csp} along the lines of Reiner, Stanton and White might be difficult because the relevant symmetric polynomials do not have positive Schur coefficients (see Section \ref{sec:other}). Perhaps the point of view in Doty and Walker \cite{dw} can be used to generalize the representation-theoretic proof in \cite{rsw}.

\section{Bounded Symmetric Polynomials}\label{sec:specialization}

In this section we develop the relationship between Theorem \ref{thm:csp} and symmetric polynomials. Though this is not strictly necessary for the proof, we believe that it is the correct context for the problem and it leads to further interesting questions.

Recall that $f(z_1,\ldots,z_n)\in\ZZ[z_1,\ldots,z_n]$ is called {\em symmetric} if for any permutation $\pi$ of $\{1,\ldots,n\}$ we have
\begin{equation*}
f(z_{\pi(1)},\ldots,z_{\pi(n)})=f(z_1,\ldots,z_n).
\end{equation*}
The following observation guides the shape of this paper. It is surely well-known, but we have not seen it explicitly stated. We do not call this a lemma because we will not use these results in later proofs. Parts (b) and (c) of the observation use some basic Galois theory; however, they also follow from Theorem \ref{thm:main} below. See Remark \ref{rem:elementary}.

\begin{observation}\label{obs:galois} Let $f(z_1,\ldots,z_n)\in\ZZ[z_1,\ldots,z_n]$ be a symmetric polynomial with integer coefficients and let $\omega$ be a primitive $d$th root of unity.
\begin{enumerate}
\item If $d|n$ then we have $\omega^n=1$, hence
\begin{equation*}
f(\omega,\omega^2,\ldots,\omega^n)=f(\omega,\omega^2,\ldots,\omega^{n-1},1)=f(1,\omega,\ldots,\omega^{n-1}).
\end{equation*} 
If, moreover, $f$ is homogeneous of degree $k$ with $d\nmid k$ then we have
\begin{equation*}
f(\omega,\omega^2,\ldots,\omega^n)=f(1,\omega,\ldots,\omega^{n-1})=0.
\end{equation*}
\item If $d|n$ or $d|(n-1)$ then $f(1,\omega,\ldots,\omega^{n-1}) \in\ZZ$.
\item If $d|n$ or $d|(n+1)$ then $f(\omega,\omega^2,\ldots,\omega^n) \in\ZZ$.
\item If $d|(n-1)$ then $f(\omega,\omega^2,\ldots,\omega^n)$ need not be an integer.
\item If $d|(n+1)$ then $f(1,\omega,\ldots,\omega^{n-1})$ need not be an integer.
\end{enumerate}
\end{observation}

\begin{proof} (a): If $d|n$ and if $f$ is homogeneous of degree $k$ then
\begin{equation*}
f(\omega,\omega^2,\ldots,\omega^n)=\omega^k f(1,\omega,\ldots,\omega^{n-1})=\omega^kf(\omega,\omega^2,\ldots,\omega^n).
\end{equation*}
If $f(\omega,\omega^2,\ldots,\omega^n)\neq 0$ then this implies that $\omega^k= 1$, hence $d| k$.

Parts (b) and (c) use the following setup from Galois theory. Let $\zeta_m$ denote a primitive $m$th root of unity. The Galois group of $\QQ(\zeta_m)/\QQ$ is $\{\varphi_r:\gcd(r,m)=1\}$, where the automorphism $\varphi_r:\QQ(\zeta_m)\to\QQ(\zeta_m)$ is defined by $\varphi_r(\zeta_m)=\zeta_m^r$. If $\alpha\in\QQ(\zeta_m)$ satisfies $\varphi_r(\alpha)=\alpha$ for all $\gcd(r,m)=1$ then it follows that $\alpha\in\QQ$. If we also have $\alpha\in\ZZ[\zeta_m]$ then $\alpha$ satisfies a monic polynomial over $\ZZ$, which implies that $\alpha\in\ZZ$.

Let $\omega$ be a primitive $d$th root of unity and let $\bomega=(\omega,\ldots,\omega^{d-1})$ and write powers to denote concatenation of sequences. If $\gcd(r,d)=1$ then the automorphism $\varphi_r(\omega)=\omega^r$ permutes the sequence $\bomega$, hence it permutes sequences of the following four types:
\begin{equation*}
(1,\bomega)^m, \quad
(\bomega,1)^m,\quad
((1,\bomega)^m,1),\quad
((\bomega,1)^m,\bomega).
\end{equation*}
(b): If $d|n$ and $\gcd(r,n)=1$ then we also have $\gcd(r,d)=1$, hence the sequence $(1,\omega,\ldots,\omega^{n-1})=(1,\bomega)^{n/d}$ is permuted by $\varphi_r$. Since $f$ is symmetric it follows that
\begin{equation*}
\varphi_r( f(1,\omega,\ldots,\omega^{n-1})) = f(\varphi_r(1,\bomega)^{n/d})=f((1,\bomega)^{n/d})=f(1,\omega,\ldots,\omega^{n-1}).
\end{equation*}
Then since $f(1,\omega,\ldots,\omega^{n-1})\in\ZZ[\zeta_n]$ we have $f(1,\omega,\ldots,\omega^{n-1})\in\ZZ$. If $d|(n-1)$ and $\gcd(r,n-1)=1$ then we also have $\gcd(r,d)=1$, hence the sequence $(1,\omega,\ldots,\omega^{n-1})=((1,\bomega)^{(n-1)/d},1)$ is permuted by $\varphi_r$. Since $f$ is symmetric, the automorphism $\varphi_r$ of $\QQ(\zeta_{n-1})$ fixes $f(1,\omega,\ldots,\omega^{n-1})$, and since $f(1,\omega,\ldots,\omega^{n-1})\in\ZZ[\zeta_{n-1}]$ this implies that $f(1,\omega,\ldots,\omega^{n-1})\in\ZZ$. (c): If $d|n$ then $(\omega,\omega^2,\ldots,\omega^n)=(\bomega,1)^{n/d}$ and if $d|(n+1)$ then $(\omega,\omega^2,\ldots,\omega^n)=((\bomega,1)^{(n+1)/d-1},\bomega)$. The rest of the proof is the same as (b).

(d): If $d|(n-1)$ and $f(z_1,\ldots,z_n)=z_1+\cdots+z_n$ then $f(\omega,\omega^2,\ldots,\omega^n)=\omega\not\in\ZZ$. (e): If $d|(n+1)$ and $f(z_1,\ldots,z_n)=z_1+\cdots+z_n$ then $f(1,\omega,\ldots,\omega^{n-1})=-\omega\not\in\ZZ$.
\end{proof}

For an indeterminate $q$ and a symmetric polynomial $f(z_1,\ldots,z_n)\in\ZZ[z_1,\ldots,z_n]$ we consider the {\em principal specialization}
\begin{equation*}
f(1,q,q^2,\ldots,q^{n-1})\in\ZZ[q].
\end{equation*}
The evaluation of $q$ at roots of unity appears often in the literature for specific families of $f$, but we have not seen it spelled out that this is a general phenomenon with three different cases: $d|n$, $d|(n-1)$ and $d|(n+1)$. The incompatibility of the cases $d|(n-1)$ and $d|(n+1)$ leads to some ambiguity whether one should define the principal specialization as $f(1,\ldots,q^{n-1})$ or $f(q,\ldots,q^n)$. The former is the standard convention.

In this paper we are interested in the case when $f$ is a certain generalization of the elementary symmetric polynomials $e_k$ and the complete symmetric polynomials $h_k$. Recall the generating functions
\begin{align*}
E(t;z_1,\ldots,z_n) &= \prod_{i=1}^n (1+z_it)= \sum_{k\geq 0} e_k(z_1,\ldots,z_n) t^k,\\
H(t;z_1,\ldots,z_n) &= \prod_{i=1}^n (1-z_it)^{-1} = \sum_{k\geq 0} h_k(z_1,\ldots,z_n) t^k,
\end{align*}
and the reciprocity relation $E(-t;z_1,\ldots,z_n)H(t;z_1,\ldots,z_n)=1$. The principal specializations of $e_k$ and $h_k$ satisfy
\begin{align}
e_k(1,q,\ldots,q^{n-1}) &= q^{k(k-1)/2} \sqbinom{n}{k}_q,\\
h_k(1,q,\ldots,q^{n-1}) &= \sqbinom{n-1+k}{n-1}_q.
\end{align}
For any $b\geq 2$ we consider the {\em $b$-bounded symmetric polynomials} $h_k^{(b)}$ with generating function
\begin{equation*}
H^{(b)}(t;z_1,\ldots,z_n) = \sum_{k\geq 0} h_k^{(b)}(z_1,\ldots,z_n)t^k = \prod_{i=1}^n (1+z_it+(z_it)^2+\cdots+(z_it)^{b-1}).
\end{equation*}
Note that $h_k^{(2)}=e_k$ and $h_k^{(b)}=h_k$ when $b>k$. We can express the generating function $H^{(b)}$ in terms of $E$ and $H$. This formula appears in several places; see, e.g., \cite[Theorem 3.12]{dw}.

\begin{lemma}\label{lem:gf}
\begin{equation*}
H^{(b)}(t;z_1,\ldots,z_n)= E(-t^b;z_1^b,\ldots,z_n^b) H(t;z_1,\ldots,z_n).
\end{equation*}
\end{lemma}
\begin{proof}
\begin{align*}
H^{(b)}(t;z_1,\ldots,z_n) & = \prod_{i=1}^{n} (1+z_it+(z_it)^2+\cdots+(z_it)^{b-1})\\
&= \prod_{i=1}^n \frac{1-(z_it)^b}{1-z_it}\\
&= \prod_{i=1}^n (1+z_i^b(-t^b)) \prod_{i=1}^n (1-z_it)^{-1}\\
&= E(-t^b;z_1^b,\ldots,z_n^b) H(t;z_1,\ldots,z_n).
\end{align*}
\end{proof}
It follows that
\begin{align*}
\sum_{k\geq 0} h_k^{(b)}(z_1,\ldots,z_n)t^k &= \sum_{\ell\geq 0} e_\ell(z_1^b,\ldots,z_n^b)(-t^b)^\ell \sum_{m\geq 0} h_m(z_1,\ldots,z_n)t^m\\
&= \sum_{k\geq 0} \left(\sum_{b\ell+m=k} (-1)^\ell e_\ell(z_1^b,\ldots,z_n^b) h_m(z_1,\ldots,z_n)\right) t^k\\
&= \sum_{k\geq 0} \left(\sum_{\ell\geq 0} (-1)^\ell e_\ell(z_1^b,\ldots,z_n^b) h_{k-b\ell}(z_1,\ldots,z_n)\right) t^k
\end{align*}
and hence
\begin{equation}\label{eq:ugly}
 h_k^{(b)}(z_1,\ldots,z_n) = \sum_{\ell\geq 0} (-1)^\ell e_\ell(z_1^b,\ldots,z_n^b) h_{k-b\ell}(z_1,\ldots,z_n),
\end{equation}
where we define $h_k=0$ for negative $k$. This result can also be proved by inclusion-exclusion. In terms of the set \eqref{eq:xset} we have
\begin{equation*}
h_k^{(b)}(z_1,\ldots,z_n) = \sum_{(x_1,\ldots,x_n)\in X_b^n[k]} z_1^{x_1}\cdots z_n^{x_n},
\end{equation*}
so that $h_k^{(b)}$ is the generating function for $k$-multisets with multiplicities bounded above by $b$. Using notation from the previous section, we have
\begin{equation*}
\binom{n}{k}^{(b)}=h_k^{(b)}(\underbrace{1,\ldots,1}_{\text{$n$ times}})=\#X_b^n[k],
\end{equation*}
where $\binom{n}{k}^{(b)}$ is the coefficient of $t^k$ in $(1+t+\cdots+t^{b-1})^n$. The principal specialization gives a natural $q$-analogue
\begin{align*}
\sqbinom{n}{k}_q^{(b)}&:=h_k^{(b)}(1,q,\ldots,q^{n-1}) \\
&= \sum_{(x_1,\ldots,x_n)\in X_b^n[k]} q^{x_2+2x_3+\cdots+(n-1)x_n}.
\end{align*}
Since $e_\ell(1,q,\ldots,q^{n-1})=q^{\ell(\ell-1)/2}\sqbinom{n}{\ell}_q$ and $h_m(1,q,\ldots,q^{n-1})=\sqbinom{n-1+m}{m}_q$, equation \eqref{eq:ugly} implies the explicit formula
\begin{align}
\sqbinom{n}{k}_q^{(b)} &= \sum_{\ell\geq 0} (-1)^\ell e_\ell\left((q^0)^b,(q^1)^b,\ldots,(q^{n-1})^{b}\right) h_{k-b\ell}(1,q,\ldots,q^{n-1}) \nonumber \\
&= \sum_{\ell\geq 0} (-1)^\ell e_\ell\left((q^b)^0,(q^b)^1,\ldots,(q^b)^{n-1}\right) h_{k-b\ell}(1,q,\ldots,q^{n-1}) \nonumber \\
&= \sum_{\ell\geq 0} (-1)^\ell (q^b)^{\ell(\ell-1)/2}\sqbinom{n}{\ell}_{q^b} \sqbinom{n-1+k-b\ell}{k-b\ell}_q,\label{eq:uglier}
\end{align}
though we will never use this formula because our proof of Theorem \ref{thm:main} below follows from more elegant generating function arguments.

The polynomials $h_k^{(b)}(z_1,\ldots,z_n)$ have been studied by several authors. Doty and Walker \cite{dw} called them {\em modular complete symmetric functions} and showed that when $b=p$ is prime they are related to mod $p$ representation theory of $GL_n$.  We take the notation $h_k^{(b)}$ from Fu and Mei \cite{fumei} who called these the {\em truncated homogeneous symmetric functions}. More recently, Grinberg \cite{grinberg} computed the Schur expansion, which he showed has coefficients in $\{0,1,-1\}$. He called $h_k^{(b)}$ a {\em Petrie symmetric function} because the Schur coefficients are determinants of certain {\em Petrie matrices} (see the end of Section \ref{sec:other}). After publicizing his results, Grinberg learned about several previous occurrences of these functions and collected an extensive list of references in the introduction to \cite{grinberg}. In private communication, he indicated that he would not have chosen the name {\em Petrie} if he had known about the previous occurrences. It seems that such a basic concept should have a basic name; in consultation with Grinberg we arrived at the name {\em $b$-bounded}. Of course, one need not use the letter $b$.

Let $\omega$ be a primitive $d$th root of unity. Since $h_k^{(b)}(z_1,\ldots,z_n)\in\ZZ[z_1,\ldots,z_n]$ is symmetric we know from Observation \ref{obs:galois} that $\sqbinom{n}{k}_\omega^{(b)}=h_k^{(b)}(1,\omega,\ldots,\omega^{n-1})=h_k^{(b)}(\omega,\omega^2,\ldots,\omega^n)\in\ZZ$ when $d|n$, $\sqbinom{n}{k}_\omega^{(b)}=h_k^{(b)}(1,\omega,\ldots,\omega^{n-1})\in\ZZ$ when $d|(n-1)$ and $\omega^k\sqbinom{n}{k}_\omega^{(b)}=h_k^{(b)}(\omega,\omega^2,\ldots,\omega^n)\in\ZZ$ when $d|(n+1)$. Our main theorem gives generating functions for these integers. The proof is not difficult but the result is quite natural and seems to have been overlooked until now. It can be viewed as a systematization and generalization of Proposition 4.2 of Reiner, Stanton and White \cite{rsw}. See Remark \ref{rem:rsw} below.

\begin{theorem}[Main Theorem]\label{thm:main}
Let $\omega$ be a primitive $d$th root of unity.  Recall that we define $\binom{n}{k}^{(b)}=0$ and $\sqbinom{n}{k}_q^{(b)}=0$ when $n$ or $k$ is not an integer.
\begin{enumerate}
\item If $d|n$ and $\gcd(b,d)=g$ then
\begin{equation*}
\sum_{k=0}^{(b-1)n} \sqbinom{n}{k}_\omega^{(b)}t^k = \left( \frac{(1-t^{bd/g})^g}{1-t^d}\right)^{n/d}.
\end{equation*}
When $g=1$ this becomes
\begin{equation*}
\sum_{k=0}^{(b-1)n} \sqbinom{n}{k}_\omega^{(b)}t^k = \left( \frac{1-(t^d)^b}{1-t^d}\right)^{n/d} = [b]_{t^d}^{n/d},
\end{equation*}
and hence
\begin{equation*}
\sqbinom{n}{k}_\omega^{(b)} = \binom{n/d}{k/d}^{(b)}.
\end{equation*}
\item If $d|(n-1)$ and $\gcd(b,d)=g$ then
\begin{equation*}
\sum_{k=0}^{(b-1)n} \sqbinom{n}{k}_\omega^{(b)}t^k =[b]_t \left(\frac{(1-t^{bd/g})^g}{1-t^d}\right)^{(n-1)/d}.
\end{equation*}
When $g=1$ this becomes
\begin{equation*}
\sum_{k=0}^{(b-1)n} \sqbinom{n}{k}_\omega^{(b)}t^k = [b]_t [b]_{t^d}^{(n-1)/d},
\end{equation*}
and hence
\begin{equation*}
\sqbinom{n}{k}_\omega^{(b)} = \sum_{\ell=0}^{b-1} \binom{(n-1)/d}{(k-\ell)/d}^{(b)}.
\end{equation*}
\item If $d|(n+1)$ and $\gcd(b,d)=g$ then
\begin{equation*}
\sum_{k=0}^{(b-1)n} \omega^k\sqbinom{n}{k}_\omega^{(b)}t^k = \frac{1}{[b]_t}\left(\frac{(1-t^{bd/g})^g}{1-t^d}\right)^{(n+1)/d}.
\end{equation*}
When $g=1$ this becomes
\begin{equation}
\sum_{k=0}^{(b-1)n} \omega^k\sqbinom{n}{k}_\omega^{(b)}t^k = \frac{1}{[b]_t}[b]_{t^d}^{(n+1)/d}.
\label{eq:nplusone}
\end{equation}
The explicit computation of the coefficients is related to the Frobenius coin problem. See Corollary \ref{cor:interesting} in the next section. For later use in Section \ref{sec:other} we simply mention the special case when $b=2$ and $d=n+1$. If $\gcd(2,n+1)=g$ then one can check that
\begin{equation*}
\frac{1}{[2]_t}\left(\frac{(1-t^{2(n+1)/g})^g}{1-t^{n+1}}\right) = 1-t+t^2-\cdots+(-1)^n t^n,
\end{equation*}
and hence
\begin{equation*}
e_k(\omega,\omega^2,\ldots,\omega^n) = \omega^k \sqbinom{n}{k}_\omega^{(2)}=\omega^{k(k+1)/2}\sqbinom{n}{k}_\omega=\begin{cases}
(-1)^k & k<n+1,\\
0 & k\geq n+1.
\end{cases}
\end{equation*}
\end{enumerate}
\end{theorem}

\begin{proof}
Let $\omega$ be a primitive $d$th root of unity and let $T$ be an indeterminate, so that $\prod_{j=0}^{d-1} (1-\omega^jT)=1-T^d$. If $\gcd(b,d)=1$, then the sequence $1,\omega^b,\ldots,\omega^{(d-1)b}$ is a permutation of $1,\omega,\ldots,\omega^{d-1}$, hence we also have $\prod_{j=0}^{d-1} (1-\omega^{jb}T)=1-T^d$. More generally, let $g=\gcd(b,d)$, so that $\omega^b$ is a primitive $(d/g)$th root of unity. Since $\gcd(d/g,b/g)=1$ this implies that $1,\omega^b,\ldots,\omega^{(d/g-1)b}$ is a permutation of $1,\omega^g,\ldots,\omega^{(d/g-1)g}$, hence
\begin{equation*}
\prod_{j=0}^{d-1} (1-\omega^{jb}T) = \left(\prod_{j=0}^{d/g-1} (1-\omega^{jb}T)\right)^g=(1-T^{d/g})^g.
\end{equation*}
Finally, if $d|N$ then we note that
\begin{equation*}
\prod_{j=0}^{N-1} (1-\omega^{jb}T) = \left(\prod_{j=0}^{d-1} (1-\omega^{jb}T)\right)^{N/d}=((1-T^{d/g})^g)^{N/d}.
\end{equation*}
Next we consider the generating function for the principal specializations of $h_k^{(b)}$:
\begin{align*}
H^{(b)}(t;1,q,\ldots,q^{n-1}) &= \sum_{k\geq 0} h_k^{(b)}(1,q,\ldots,q^{n-1}) t^k\\
&= \prod_{j=0}^{n-1} (1+q^{j}t+\cdots+(q^{j}t)^{(b-1)})\\
&= \prod_{j=0}^{n-1} \frac{1-q^{jb}t^b}{1-q^jt}.
\end{align*}
(a): If $d|n$ then taking $T=t^b$ gives
\begin{equation*}
H^{(b)}(t;1,\omega,\ldots,\omega^{n-1}) = \frac{\prod_{j=0}^{n-1} (1-\omega^{jb}(t^b))}{\prod_{j=0}^{n-1} (1-\omega^jt)} = \frac{(1-(t^b)^{d/g})^{gn/d}}{(1-t^d)^{n/d}}=\left( \frac{(1-t^{bd/g})^g}{1-t^d}\right)^{n/d}.
\end{equation*}
(b): For the case $d|(n-1)$ we first observe that
\begin{equation*}
H^{(b)}(t;1,q,\ldots,q^{n-1}) = \frac{1-q^{(n-1)b}t^b}{1-q^{n-1}t} \, \prod_{j=0}^{n-2} \frac{1-q^{jb}t^b}{1-q^jt} = [b]_{q^{n-1}t}  \prod_{j=0}^{n-2} \frac{1-q^{jb}t^b}{1-q^jt}.
\end{equation*}
If $d|(n-1)$ then we have $[b]_{\omega^{n-1}t}=[b]_t$. So taking $N=n-1$ and $T=t^b$ gives
\begin{equation*}
H^{(b)}(t;1,\omega,\ldots,\omega^{n-1}) = [b]_{t}  \prod_{j=0}^{n-2} \frac{1-\omega^{jb}t^b}{1-\omega^jt} = [b]_t \left( \frac{(1-t^{bd/g})^g}{1-t^d}\right)^{(n-1)/d}.
\end{equation*}
(c): For the case $d|(n+1)$ we first observe that
\begin{equation*}
H^{(b)}(t;q,\ldots,q^n) = \prod_{j=1}^{n} \frac{1-q^{jb}t^b}{1-q^jt}=\frac{1}{[b]_{q^{n+1}t}}\prod_{j=0}^{n} \frac{1-q^{jb}t^b}{1-q^jt}.
\end{equation*}
If $d|(n+1)$ then we have $[b]_{\omega^{n+1}t}=[b]_t$. So taking $N=n+1$ and $T=t^b$ gives
\begin{equation*}
H^{(b)}(t;\omega,\omega^2,\ldots,\omega^n) = \frac{1}{[b]_t} \prod_{j=0}^{n} \frac{1-\omega^{jb}t^b}{1-\omega^jt} = \frac{1}{[b]_t}\left( \frac{(1-t^{bd/g})^g}{1-t^d}\right)^{(n+1)/d}.
\end{equation*}
\end{proof}

\begin{remark}\label{rem:rsw} Parts (a) and (b) of Theorem \ref{thm:main} generalize and explain Proposition 4.2 in Reiner, Stanton and White \cite{rsw}. Their identities 4.2(i) and 4.2(ii) correspond to part (a) with $b=\infty$ and $b=2$, respectively. Their 4.2(iii) and 4.2(iv) correspond to part (b) with $b=\infty$ and $b=2$, respectively. Part (c) has no analogue in \cite{rsw}.
\end{remark}

\begin{remark}\label{rem:elementary} Special cases of this theorem give a combinatorial proof of Observation \ref{obs:galois} that avoids Galois theory. That is, in the case $b=2$ this theorem shows that the appropriate specializations of $h_k^{(2)}=e_k$ are integers. (When $d$ is odd this involves some sign bookkeeping, so perhaps it is more natural to use the case $b=\infty$ when $h_k^{(\infty)}=h_k$.) It is well known that any symmetric polynomial in $\ZZ[z_1,\ldots,z_n]$ can be expressed (uniquely) as a polynomial in $e_0,e_1,\ldots,e_n$ (also as a polynomial in $h_0,h_1,\ldots,h_n$) with integer coefficients.
\end{remark}

Our original proof of Theorem \ref{thm:main}a was based on the explicit formula \eqref{eq:uglier} and involved lots of annoying case by case analysis, using the identities from \cite[Prop 4.2]{rsw}. A similar brute force proof of Theorem \ref{thm:main}b could be carried out along the same lines. It seems more difficult to prove Theorem \ref{thm:main}c without using generating functions.

In the next section we will study the coefficients in equation \eqref{eq:nplusone}. To end this section we mention an interesting corollary of Theorem \ref{thm:main}.

\begin{corollary}\label{cor:tequalsone}
Let $\omega$ be a primitive $d$th root of unity.
\begin{enumerate}
\item If $d|n$ then we have
\begin{equation*}
\sum_{k=0}^{(b-1)n} \sqbinom{n}{k}_\omega^{(b)} = \begin{cases} b^{n/d} & \gcd(b,d)=1,\\ 0 & \gcd(b,d)\neq 1.\end{cases}
\end{equation*}
\item If $d|(n-1)$ then we have
\begin{equation*}
\sum_{k=0}^{(b-1)n} \sqbinom{n}{k}_\omega^{(b)} = \begin{cases} b^{(n-1)/d+1} & \gcd(b,d)=1,\\ 0 & \gcd(b,d)\neq 1.\end{cases}
\end{equation*}
\item If $d|(n+1)$ then we have
\begin{equation*}
\sum_{k=0}^{(b-1)n} \omega^k\sqbinom{n}{k}_\omega^{(b)} = \begin{cases} b^{(n+1)/d-1} & \gcd(b,d)=1,\\ 0 & \gcd(b,d)\neq 1.\end{cases}
\end{equation*}
\end{enumerate}
\end{corollary}

\begin{proof} For $\gcd(b,d)=1$ this follows easily by setting $t=1$ in Theorem \ref{thm:main}. For $\gcd(b,d)=g\neq 1$ we use L'Hospital's rule to observe that
\begin{equation*}
\lim_{t\to 1} \frac{(1-t^{bd/g})^g}{1-t^d} = \lim_{t\to 1} \frac{\frac{d}{dt} (1-t^{bd/g})^g}{\frac{d}{dt}(1-t^d)} = \lim_{t\to 1} b (1-t^{bd/g})^{g-1} t^{d(b-g)/g} = 0.
\end{equation*}
\end{proof}

The case $\gcd(b,d)=1$ of parts (a) and (b) also follows bijectively from Theorem \ref{thm:csp}. To see this for part (a), note from Theorem \ref{thm:csp}a that the sum in Corollary \ref{cor:tequalsone}a is the size of the set $(X_b^n)^{\rho^{n/d}}$ of all words $(x_1,\ldots,x_n)\in\{0,1,\ldots,b-1\}^n$ fixed by the rotation $\rho^{n/d}$. But we have a bijection $X_b^{n/d}\to(X_b^n)^{\rho^{n/d}}$ sending
\begin{equation*}
(x_1,\ldots,x_{n/d})\mapsto (x_1,\ldots,x_{n/d},\ldots,x_1,\ldots,x_{n/d}),
\end{equation*}
which shows that $\#(X_b^n)^{\rho^{n/d}}=b^{n/d}$.
Part (c) is related to ``$q$-analogues'' of some number-theoretic congruences. Note that
\begin{align*}
\sum_{k=0}^{(b-1)n} q^k \sqbinom{n}{k}_q^{(b)} t^k &= \sum_{k=0}^{(b-1)n} \sqbinom{n}{k}_q^{(b)} (qt)^k\\
&= \prod_{i=0}^{n-1} (1+q^{i+1}t+\cdots+(q^{i+1}t)^{b-1})\\
&= \prod_{i=1}^n [b]_{q^i t}.
\end{align*}
If $\omega$ is a primitive $d$th root of unity then substituting $q=\omega$ is the same as taking the equivalence class modulo the cyclotomic polynomial $\Phi_d(q)$. If $\gcd(b,d)=1$ and $d|(n+1)$ then Corollary \ref{cor:tequalsone}c becomes
\begin{equation*}
\prod_{i=1}^n [b]_{q^i} \equiv b^{(n+1)/d-1}  \text{ mod } \Phi_d(q).
\end{equation*}
If $p$ is prime then the $p$th cyclotomic polynomial is $\Phi_p(q)=1+q+\cdots+q^{p-1}=[p]_q$. If $n+1=d=p$ and $p\nmid b$ then this becomes
\begin{equation*}
\prod_{i=1}^{p-1} [b]_{q^i} \equiv 1 \text{ mod } [p]_q,
\end{equation*}
which Pan and Sun \cite[Formula 1.6]{ps} call a ``$q$-analogue'' of Fermat's little theorem. Other instances of Corollary \ref{cor:tequalsone}c may be interesting in this context.

\section{Sylvester Coin Polynomials}\label{sec:sylvester}

We have seen that parts (a) and (b) of Theorem \ref{thm:main} are related to cyclic sieving of multisets. In this section we investigate the combinatorial meaning of the integers in Theorem \ref{thm:main}c. If $\gcd(b,d)=1$ then it follows indirectly from Theorem \ref{thm:main}c that the expression $[b]_{t^d}/[b]_t$ is a polynomial with integer coefficients. While investigating this polynomial we discovered a relation to the Frobenius coin problem, which we now describe.

Given a list of positive integers $\a=(a_1,\ldots,a_m)\in\NN\setminus\{0\}^m$ and a natural number $n\in\NN$, let $\nu_\a(n)$ be the number of ways to express $n$ as a linear combination of $a_1,\ldots,a_m$ with non-negative coefficients:
\begin{equation*}
\nu_\a(n) := \#\{(k_1,\ldots,k_m)\in\NN^m: k_1a_1+\cdots +k_ma_m=n\}.
\end{equation*}
In terms of generating functions we have
\begin{equation*}
\sum_{n\geq 0} \nu_\a(n) t^n = \frac{1}{\prod_{i=1}^m (1-t^{a_i})}.
\end{equation*}
The numbers $\nu_\a(n)$ are associated with the names Frobenius and Sylvester. Sylvester \cite{sylvester} called these numbers ``denumerants''. Alfred Brauer \cite{brauer} attributed the general problem to Frobenius, ``who mentioned it occasionally in his lectures''. Today it is called the Frobenius coin problem, or the Diophantine Frobenius problem. See \cite{alfonsin} for a survey.

The first theorem in this subject says that if $\gcd(\a)=1$ then there exists a largest $n$ with $\nu_\a(n)=0$. (Alfons\'in \cite[Theorem 1.0.1]{alfonsin} calls this a ``folk result''.) When $\gcd(\a)=1$ we will denote the finite set of non-representable numbers by
\begin{equation*}
S_\a := \{n\in\NN\setminus\{0\}: \nu_\a(n)=0\}.
\end{equation*}
The letter $S$ is for Sylvester, who considered this set in the case of two parameters, and proved that $\#S_{a_1,a_2}=(a_1-1)(a_2-1)/2$. We call the generating polynomial for these numbers the {\em Sylvester coin polynomial}:
\begin{equation*}
S_\a(t):= \sum_{s\in S_\a} t^s.
\end{equation*}
For example, we have $S_{3,5}(t)=t+t^2+t^4+t^7$. The general Frobenius problem is very difficult. In this paper we are interested in the easiest, but still quite interesting, case of two coprime parameters. We will call these parameters $(a_1,a_2)=(b,d)$ to agree with the notation in Theorem \ref{thm:main}.

The main result connecting our Theorem \ref{thm:main}c to the Frobenius coin problem is the following. Brown and Shiue \cite{bs} attribute this result to \"Ozl\"uk. Krattenthaler and M\"uller \cite[Lemma 5]{km} studied the polynomial $[b]_{t^d}/[b]_t$ in the context of cyclic sieving, but did not notice the connection to the Frobenius problem.

\begin{lemma}\label{lem:sylvester}
Let $\gcd(b,d)=1$ and consider the finite set $S_{b,d}\subseteq\NN$. Then we have
\begin{equation*}
S_{b,d}(t) = \frac{t^{bd}-1}{(1-t^b)(1-t^d)} + \frac{1}{1-t},
\end{equation*}
and hence
\begin{equation*}
1+(t-1)S_{b,d}(t) = 1+\left(\frac{(1-t^{bd})(t-1)}{(1-t^b)(1-t^d)}-1\right) = \frac{1-(t^d)^b}{1-t^d}\frac{1-t}{1-t^b}= \frac{[b]_{t^d}}{[b]_t}.
\end{equation*}
By symmetry we also have $[d]_{t^b}/[d]_t=1+(t-1)S_{b,d}(t)$.
\end{lemma}
For example, when $(b,d)=(7,5)$ we have
\begin{align*}
S_{7,5} &=   \left\{ 1,2,3,4,6,8,9,11,13,16,18,23 \right\},\\
S_{7,5}(t) &= t+{t}^{2}+{t}^{3}+{t}^{4}+{t}^{6}+{t}^{8}+{t}^{9}+{t}^{11}+{t}^{13}+{t
}^{16}+{t}^{18}+{t}^{23}
,\\
[7]_{t^5}\,/\,[7]_t &= 1+(t-1)S_{7,5}(t)\\
&= 1+{t}^{5}+{t}^{7}+{t}^{10}+{t}^{12}+{t}^{14}+{t}^{17}+{t}^{19}+{t}^{24
}
\\
&\quad -t-{t}^{6}-{t}^{8}-{t}^{11}-{t}^{13}-{t}^{16}-{t}^{18}-{t}^{23}.
\end{align*}
Note that the coefficients of $[7]_{t^5}/[7]_t$ are in $\{0,1,-1\}$. See Theorem \ref{thm:interesting}a below for a nice interpretation of these coefficients. Though Lemma \ref{lem:sylvester} is not new, we will give a complete proof because it will allow us to set up notation (the ``double abacus'') that we need for the proof of Theorem \ref{thm:interesting}. We will obtain Lemma \ref{lem:sylvester} as a corollary of Lemma \ref{lem:abacus}.

Fix $\gcd(b,d)=1$ and consider the labeling $\lambda:\ZZ^2\to\ZZ$ defined by
\begin{equation*}
\lambda(x,y):=xd+yb.
\end{equation*}
We note that $\lambda$ increases by $d$ when we take a step to the right and increases by $b$ when we take a step up. We can also think of $\lambda(x,y)$ as the ``height'' of the point $(x,y)$ relative to the line $xd+yb=0$. We are interested in the discrete half-plane of points with non-negative labels, which we call the {\em double abacus}:\footnote{The double abacus was introduced in the theory of simultaneous core partitions. See \cite{anderson} and \cite{so}.}
\begin{equation*}
A:=\{(x,y)\in\ZZ^2: xd+yb\geq 0\}.
\end{equation*}
Figure \ref{fig:abacus} shows the double abacus when $(b,d)=(7,5)$. The proofs in this section are easiest to understand visually, so we advise the reader always to keep an eye on this figure.

\begin{figure}
\begin{center}
\includegraphics[scale=1]{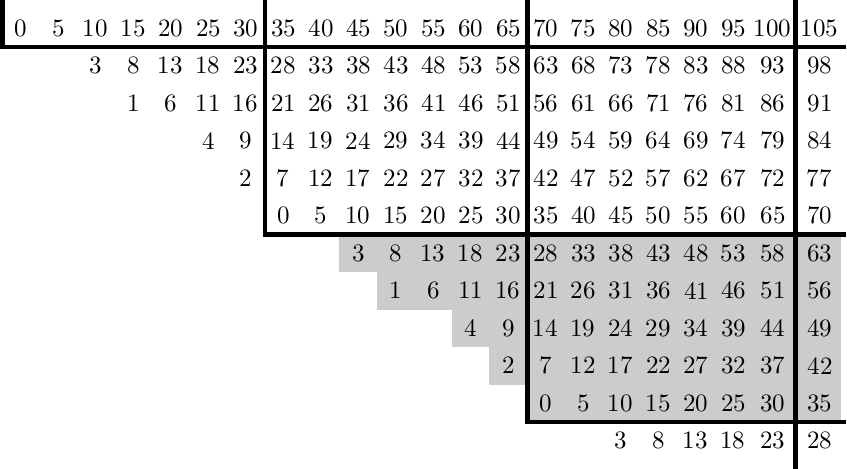}
\end{center}
\caption{The double abacus for $(b,d)=(7,5)$}
\label{fig:abacus}
\end{figure}

Now we describe a certain decomposition of the double abacus, which is indicated in the figure by black lines. Note that the labeling $\lambda$ is not injective. Indeed, for any $(x,y)\in \ZZ^2$ and $k\in\ZZ$ we have $\lambda(x-bk,y+dk)=\lambda(x,y)$. We can take as a fundamental domain any $d$ consecutive rows (or any $b$ consecutive columns). Let $R_k=\{(x,y)\in A: y=k\}$ be the $k$th row of $A$ and define the {\em fundamental row abacus}:
\begin{equation*}
R := R_0\cup R_1\cup \ldots \cup R_{d-1}.
\end{equation*}
This is the shaded region in Figure \ref{fig:abacus}. Note that the leftmost label in row $k$ is $\delta(k)$, where $\delta(k)\in\{0,1,\ldots,d-1\}$ satisfies $\delta(k)\equiv bk$ mod $d$. Thus $\lambda$ gives a bijection from the set of points $R_k$ to the set of labels $\delta(k)+d\NN$. Since $\gcd(b,d)=1$ we observe that $\delta$ defines a permutation of $\{0,1,\ldots,d-1\}$, hence each label $n\in\NN$ occurs exactly once in the row abacus $R$, and we obtain a bijection
\begin{equation*}
\lambda: R\to\NN.
\end{equation*}
The double abacus $A$ is the disjoint union of the translated row abaci $R+k(-b,d)$ for all $k\in\ZZ$. We find it useful to consider a further decomposition by columns. For each $k\geq 0$ we consider the $k$th {\em block} of $R$:\footnote{This is not related to the representation-theoretic meaning of ``block'' used in \cite{so}.}
\begin{equation*}
B_k := \{(x,y)\in R: (k-1)b\leq x<kb\}. 
\end{equation*}
Thus $R$ is the disjoint union of $B_0, B_1, B_2,\ldots$, where for each $k\geq 1$ the block $B_k$ has the shape of a $b\times d$ rectangle and the zeroth block $B_0$ has the shape of a triangle. The translated sets $B_k+\ell(-b,d)$ with $k\geq 0$ and $\ell\in\ZZ$ give a decomposition of the whole double abacus into disjoint rectangles and triangles, which is shown in Figure \ref{fig:abacus}. Observe that the labels inside the zeroth block $B_0$ are just the Sylvester set of non-representable numbers:
\begin{equation*}
S_{7,5}= \left\{ 1,2,3,4,6,8,9,11,13,16,18,23 \right\}.
\end{equation*}
The following result generalizes this fact.

\begin{lemma}\label{lem:abacus}
Fix $\gcd(b,d)=1$ and recall that we have a bijection $\lambda:R\to\NN$ from the fundamental row abacus to the set of all natural numbers. For each $n\in\NN$ recall that $\nu_{b,d}(n)$ is the number of distinct representations $n=kb+\ell d$ with $(k,\ell)\in\NN^2$. Then $\nu_{b,d}(n)=k$ if and only if the label $n$ occurs in the $k$th block of $R$:
\begin{equation*}
\nu_{b,d}(n)=k \quad\Longleftrightarrow\quad \lambda^{-1}(n)\in B_k.
\end{equation*}
If $n\geq bd$ then it follows from this that $\nu_{b,d}(n-bd)=\nu_{b,d}(n)-1$.
\end{lemma}

\begin{proof}
For the first statement, fix a point $(x,y)\in A$ in the double abacus and let $n=\lambda(x,y)=xd+yb$ be its label.  Consider the following two sets:
\begin{align*}
X&= \{(k,\ell)\in\NN^2: kd+\ell b=n\},\\
Y&= \{(x',y')\in A: x'\leq x,\,\, y'\leq y \text{ and } \lambda(x',y')=0 \}.
\end{align*}
The set $X$ consists of all representations of $n$ as an $\NN$-linear combination of $b$ and $d$, so that $\#X=\nu_{b,d}(n)$. The set $Y$ consists of all points in the double abacus that are below and to the left of $(x,y)$, and have label $0$. It is easy to check that the maps $(k,\ell)\mapsto (x-k,y-\ell)$ and $(x',y')\mapsto (x-x',y-y')$ are inverse bijections between $X$ and $Y$, hence $\#Y=\nu_{b,d}(n)$. Now the result follows from the observation that each point in the $k$th block $B_k$ has exactly $k$ points below and to its left with label $0$.

For the second part of the lemma, suppose that $n\geq bd$. Since $\lambda:R\to\NN$ is a bijection there exist unique points $(x,y), (x',y')\in R$ with $\lambda(x,y)=n$ and $\lambda(x',y')=n-bd$. But $\lambda(x-b,y)=(x-b)d+yb=(xd+yb)-bd=n-bd$ so we must have $(x',y')=(x-b,y)$. To complete the proof we observe that
\begin{align*}
\nu_{b,d}(n)=k &\quad\Longleftrightarrow \quad (x,y)\in B_k\\
& \quad\Longleftrightarrow \quad(x-b,y)\in B_{k-1} \\
& \quad\Longleftrightarrow \quad \nu_{b,d}(n-bd)=k-1.
\end{align*}
\end{proof}

For example, consider the double abacus in Figure \ref{fig:abacus} with $(b,d)=(7,5)$. The label $n=41$ occurs at position $(x,y)=(4,3)$ in block $B_1$. The unique point $(x',y')$ with $x'\leq x$, $y'\leq y$ and $\lambda(x',y')=0$ is $(x',y')=(0,0)$, which corresponds to the unique representation $41=5(x-x')+7(y-y')=5(4)+7(3)$. Next consider the label $n=88$ at position $(x,y)=(5,9)$ which is in a translation of block $B_2$. (The bijection in the proof of Lemma \ref{lem:abacus} is invariant under translation by $(-b,d)$.) The two relevant points $(x',y')$ with $x'\leq x$, $y'\leq y$ and $\lambda(x',y')=0$ are $(0,0)$ and $(-7,5)$, which correspond to the representations $88=5(5-0)+7(9-0)=5(5)+7(9)$ and $88=5(5-(-7))+7(9-5)=5(12)+7(4)$.

Now we can prove Lemma \ref{lem:sylvester}.

\begin{proof}
Fix $\gcd(b,d)=1$ and note that
\begin{equation*}
\frac{1}{(1-t^b)(1-t^d)} = (1+t^b+t^{2b}+\cdots)(1+t^d+t^{2d}+\cdots) = \sum_{n\geq 0} \nu_{b,d}(n)t^n.
\end{equation*}
Applying Lemma \ref{lem:abacus} gives
\begin{align*}
\frac{t^{bd}-1}{(1-t^b)(1-t^d)} &= \sum_{n\geq 0} \nu_{b,d}(n) t^{n+bd} - \sum_{n\geq 0} \nu_{b,d}(n) t^n\\
&= \sum_{n\geq bd} \nu_{b,d}(n-bd) t^n - \sum_{n\geq 0} \nu_{b,d}(n) t^n\\
&= -\sum_{n=0}^{bd-1}\nu_{b,d}(n)t^n + \sum_{n\geq bd} [\nu_{b,d}(n-bd)-\nu_{b,d}(n)]t^n\\
&= -\sum_{n=0}^{bd-1} \nu_{b,d}(n)t^n - \sum_{n\geq bd} t^n,
\end{align*}
and hence
\begin{equation*}
\frac{t^{bd}-1}{(1-t^b)(1-t^d)}+\frac{1}{1-t} = \sum_{n=0}^{bd-1} [1-\nu_{b,d}(n)]t^n.
\end{equation*}
But every number $0\leq n\leq bd-1$ satisfies $\nu_{b,d}(n)=0$ or $\nu_{b,d}(n)=1$, and every number $n\in\NN$ with $\nu_{b,d}(n)=0$ satisfies $n\leq bd-1$. Hence
\begin{equation*}
\frac{t^{bd}-1}{(1-t^b)(1-t^d)}+\frac{1}{1-t} = \sum_{\substack{ n\in\NN \\\nu_{b,d}(n)=0}} t^n =S_{b,d}(t).
\end{equation*}
\end{proof}

We will use this result to give two different explicit formulas (Theorem \ref{thm:alternating} and Corollary \ref{cor:interesting}) for the coefficients in Theorem \ref{thm:main}c. Recall that $\binom{n}{k}^{(b)}$ is the coefficient of $t^k$ in $[b]_t^n$ and $q^k\sqbinom{n}{k}_q^{(b)}$ is the coefficient of $t^k$ in $[b]_q[b]_{q^2}\cdots[b]_{q^n}$. If $\omega$ is a primitive $d$th root of unity with $d|(n+1)$ and if $\gcd(b,d)=1$ then Theorem \ref{thm:main}c says that
\begin{equation*}
\sum_{k=0}^{(b-1)n} \omega^k\sqbinom{n}{k}_\omega^{(b)} t^k=\frac{1}{[b]_t}[b]_{t^d}^{(n+1)/d}.
\end{equation*}
Combining this with Lemma \ref{lem:sylvester} gives the following.

\begin{theorem}\label{thm:alternating} Let $\gcd(b,d)=1$ and recall that $S_{b,d}$ is the finite set of integers $s\geq 1$ that cannot be represented as $s=bk+d\ell$ with $k,\ell\geq 0$. If $\omega$ is a primitive $d$th root of unity with $d|(n+1)$ then we have
\begin{equation*}
\omega^k\sqbinom{n}{k}_\omega^{(b)} = \binom{(n+1)/d-1}{k/d}^{(b)} + \sum_{s\in S_{b,d}} \binom{(n+1)/d-1}{(k-1-s)/d}^{(b)} - \sum_{s\in S_{b,d}} \binom{(n+1)/d-1}{(k-s)/d}^{(b)},
\end{equation*}
where we define $\binom{u}{v}^{(b)}=0$ when $u$ or $v$ is not an integer.
\end{theorem}

\begin{proof}
From Theorem \ref{thm:main}c and Lemma \ref{lem:sylvester} we have
\begin{align*}
\sum_{k=0}^{(b-1)n} \omega^k\sqbinom{n}{k}_\omega^{(b)} t^k &= \frac{[b]_{t^d}}{[b]_t} [b]_{t^d}^{(n+1)/d-1}\\
&= \left( 1+(t-1) S_{b,d}(t) \right)  [b]_{t^d}^{(n+1)/d-1}\\
&= [b]_{t^d}^{(n+1)/d-1} + (t-1) S_{b,d}(t) [b]_{t^d}^{(n+1)/d-1}\\
&=  \sum_{\ell} \binom{(n+1)/d-1}{\ell}^{(b)} t^{\ell d}  + (t-1) \sum_{s\in S_{b,d}} t^s \sum_{\ell} \binom{(n+1)/d-1}{\ell}^{(b)} t^{\ell d}.
\end{align*}
Taking the coefficient of $t^k$ gives the result.
\end{proof}

It is not clear from this alternating formula when the coefficients are zero, positive or negative. Corollary \ref{cor:interesting} will give an explicit criterion for this. It is based on the following theorem, which we believe has independent interest.

\begin{theorem}\label{thm:interesting} Let $\gcd(b,d)=1$. Let $(\delta_0,\ldots,\delta_{d-1})$ be the permutation of $(0,\ldots,d-1)$ defined by $\delta_r b\equiv r$ mod $d$ and let $(\beta_0,\ldots,\beta_{b-1})$ be the permutation of $(0,\ldots,b-1)$ defined by $\beta_r d\equiv r$ mod $b$. 
\begin{enumerate}
\item We have the following formula that is symmetric in $b$ and $d$:
\begin{equation*}
\frac{[b]_{t^d}}{[b]_t} =\frac{[d]_{t^b}}{[d]_t}= [\beta_1]_{t^d}[\delta_1]_{t^b} -t [b-\beta_1]_{t^d} [d-\delta_1]_{t^b}.
\end{equation*}
The coefficients of $[\beta_1]_{t^d}[\delta_1]_{t^b}$ are in $\{0,1\}$, the coefficients of $-t [b-\beta_1]_{t^d} [d-\delta_1]_{t^b}$ are in $\{0,-1\}$, and there is no cancellation. That is, there are $\beta_1\delta_1$ coefficients equal to $1$ and $(b-\beta_1)(d-\delta_1)$ coefficients equal to $-1$.
\item We also have the following non-symmetric formula for the $d$-multisections of $[b]_{t^d}/[b]_t$. For all $r\in\{0,1,\ldots,d-1\}$ define $\gamma_r:=(\delta_rb-r)/d \in\{0,\ldots,b-1\}$. Then we have
\begin{equation*}
\frac{[b]_{t^d}}{[b]_t}  = \sum_{r=0}^{d-1} t^r f_r(t^d),
\end{equation*}
where
\begin{equation*}
f_r(t) = \begin{cases}
t^{\gamma_r}[\beta_1]_t & \delta_r<\delta_1,\\
-t^{\gamma_r-b+\beta_1} [b-\beta_1]_t & \delta_r\geq \delta_1.
\end{cases}
\end{equation*}
\end{enumerate}
\end{theorem}

Before giving the proof, we would like to indicate that the expression in Theorem \ref{thm:interesting}a has a simple geometric meaning. Namely, let $X_0\subseteq R$ be the $\beta_1\times\delta_1$ rectangle of points in the fundamental row abacus whose bottom left point has label $0$, and let $X_1\subseteq R$ be the $(b-\beta_1)\times (d-\delta_1)$ rectangle of points whose bottom left point has label $1$. Because of the way that the labels $\lambda$ are defined we have
\begin{equation*}
\sum_{\x\in X_0} t^{\lambda(\x)} = [\beta_1]_{t^d}[\delta_1]_{t^b} \quad \text{ and } \quad \sum_{\x\in X_1} t^{\lambda(\x)} =  t[b-\beta_1]_{t^d} [d-\delta_1]_{t^b}.
\end{equation*}
Since the labeling $\lambda:R\to\NN$ is injective and since the rectangles $X_0,X_1$ are disjoint, there can be no cancellation between these two polynomials.

For example, consider Figure \ref{fig:abacus} where $(b,d)=(7,5)$. The relevant permutations are $(\beta_0,\ldots,\beta_6)=(0,3,6,2,5,1,4)$ and $(\delta_0,\ldots,\delta_4)=(0,3,1,4,2)$, hence $\beta_1=\delta_1=3$. The $\beta_1\times\delta_1=3\times 3$ rectangle $X_0$ contains the labels
\begin{equation*}
\lambda(X_0) = \{0,5,10, 7,12, 17, 14,19,24\},
\end{equation*}
while the rectangle $X_1$ of shape $(b-\beta_1)\times(d-\delta_1)=4\times 2$ contains the labels
\begin{equation*}
\lambda(X_1) = \{1,6,11,16 ,8,13,18,23\}.
\end{equation*}
Hence
\begin{align*}
[7]_{t^5}/[7]_t & = \sum_{\x\in X_0} t^{\lambda(\x)} - \sum_{\x \in X_1} t^{\lambda(\x)} \\
&=  [\beta_1]_{t^d}[\delta_1]_{t^b} -t [b-\beta_1]_{t^d} [d-\delta_1]_{t^b}\\
&= [3]_{t^5}[3]_{t^7} - t [4]_{t^5}[2]_{t^7}\\
&= 1+{t}^{5}+{t}^{7}+{t}^{10}+{t}^{12}+{t}^{14}+{t}^{17}+{t}^{19}+{t}^{24
}
\\
&\quad -t-{t}^{6}-{t}^{8}-{t}^{11}-{t}^{13}-{t}^{16}-{t}^{18}-{t}^{23}.
\end{align*}
We will give an example of Theorem \ref{thm:interesting}b at the end of this section.

\begin{proof}
We will apply Lemma \ref{lem:sylvester} and the double abacus. Throughout the proof we advise the reader to refer to Figure \ref{fig:two_triangles}. The key idea is that for any point $(x,y)\in A$ and $k,\ell\in\NN$, the generating polynomial for labels in the $k\times \ell$ rectangle $\{(x+k',y+\ell'): 0\leq k'\leq k-1, 0\leq \ell'\leq \ell-1\}$ is
\begin{align*}
\sum_{\substack{0\leq k'\leq k-1 \\ 0\leq \ell' \leq \ell-1}} t^{\lambda(x+k',y+\ell')} &= \sum_{\substack{0\leq k'\leq k-1 \\ 0\leq \ell' \leq \ell-1}} t^{(x+k')d+(y+\ell')b}\\
&= t^{xd+yb} \sum_{\substack{0\leq k'\leq k-1 \\ 0\leq \ell' \leq \ell-1}} t^{k'd+\ell 'b}\\
&= t^{xd+yb} \sum_{0\leq k'\leq k-1} t^{k'd} \sum_{0\leq \ell'\leq \ell-1} t^{\ell 'b}\\
&= t^{xd+yb} [k]_{t^d} [\ell]_{t^b}.
\end{align*}
If, furthermore, we have $k\leq b$ and $\ell\leq d$ then since $\gcd(b,d)=1$ the points of the $k\times \ell$ rectangle have distinct labels, hence this polynomial has coefficients in $\{0,1\}$.

\begin{figure}
\hspace{1in}\includegraphics{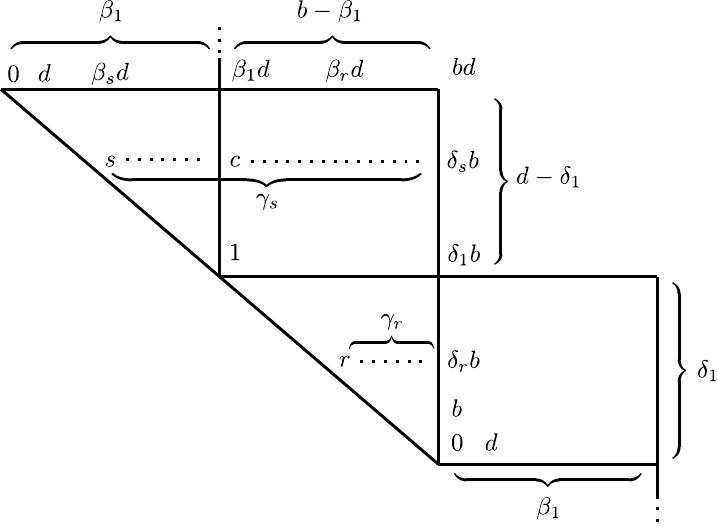}
\caption{The configuration in the proof of Theorem \ref{thm:interesting}}
\label{fig:two_triangles}
\end{figure}

We will prove part (a) by showing that $[b]_{t^d}/[b]_t$ has the form
\begin{equation*}
[b]_{t^d}/[b]_t = \sum_{\x\in X_0} t^{\lambda(\x)} - \sum_{\x\in X_1} t^{\lambda(\x)},
\end{equation*}
where $X_0$ is a $\beta_1\times\delta_1$ rectangle whose bottom left point has label $0$, and $X_1$ is a $(b-\beta_1)\times (d-\delta_1)$ rectangle disjoint from $X_0$ whose bottom left point has label $1$. To show this, we recall that the Sylvester coin polynomial has the form
\begin{equation*}
\sum_{\x\in B_0} t^{\lambda(\x)},
\end{equation*}
where the ``zeroth block'' $B_0$ is the triangular set of points weakly to the left and below the point $(-1,d-1)$. From the definition of the numbers $\beta_i$ and $\delta_i$ we note that the label $1$ occurs at the lowest point of the column containing label $\beta_1 d$ and at the leftmost point of the row containing label $\delta_1 b$. That is, the label $1$ occurs at point $(-b+\beta_1,\delta_1)$. Since $\lambda(b,-d)=0$ we also have $\lambda(\beta_1,\delta_1-d)=\lambda((-b+\beta_1)+b,(\delta_1)-d)=\lambda(-b+\beta_1,\delta_1)=1$. Define the shifted triangle
\begin{equation*}
B_0':= B_0+(\beta_1,\delta_1-d) = \{(x,y)+(\beta_1,\delta_1-d): (x,y)\in B_0\}.
\end{equation*}
We note that $B_0'$ contains every point of the triangle weakly to the left and below the point $(-1,d-1)+(\beta_1,\delta_1-d)=(\beta_1-1,\delta_1-1)$, {\bf except for the point} $(0,0)$. Indeed, since
\begin{equation*}
\lambda(x+\beta_1,y+\delta_1-d)=\lambda(x,y)+\lambda(\beta_1,\delta_1-d)=\lambda(x,y)+1,
\end{equation*}
we see that each label in $B_0'$ is one larger than the corresponding label in $B_0$, so the label $0$ does not occur. Since the labels in $B_0$ are the Sylvester set we know from Lemma \ref{lem:sylvester} that
\begin{align*}
[b]_{t^d}/[b]_t &= 1+(t-1)\sum_{\x\in B_0} t^{\lambda(\x)}\\
&= 1+\sum_{\x\in B_0} t^{\lambda(\x)+1}-\sum_{\x\in B_0} t^{\lambda(\x)}\\
&= 1+\sum_{\x\in B_0'} t^{\lambda(\x)} - \sum_{\x\in B_0} t^{\lambda(\x)}\\
&= \sum_{\x\in B_0'\cup\{(0,0)\}} t^{\lambda(\x)} - \sum_{\x\in B_0} t^{\lambda(\x)}\\
&= \sum_{\x\in B_0''} t^{\lambda(\x)} - \sum_{\x\in B_0} t^{\lambda(\x)},
\end{align*}
where $B_0'':=B_0'\cup\{(0,0)\}$. This can be simplified by observing that the triangular shapes $B_0$ and $B_0''$ are not disjoint. Their intersection is the smaller triangle $T_1$ below and to the left of the point $(-1,\delta_1-1)$. (This is the triangle containing the label $r$ in Figure \ref{fig:two_triangles}.) Hence
\begin{align*}
[b]_{t^d}/[b]_t = \sum_{\x\in B_0''\setminus T_1} t^{\lambda(\x)} - \sum_{\x\in B_0\setminus T_1} t^{\lambda(\x)}.
\end{align*}
This can be simplified still further. Let $T_2$ be the triangle below and to the left of the point $(-b+\beta_1-1,d-1)$. (This is the triangle containing the label $s$ in Figure \ref{fig:two_triangles}.) Let $X_1$ be the rectangle of size $(b-\beta_1)\times (d-\delta_1)$ with bottom left corner $(-b+\beta_1,\delta_1)$ and let $X_0$ be the $\beta_1\times\delta_1$ rectangle with bottom left corner $(0,0)$. Then we have $B_0=T_1\cup X_1\cup T_2$ and $B_0''=T_1\cup X_0\cup (T_2+(b,-d))$. But the labels in the triangles $T_2$ and $T_2+(b,-d)$ are the same. Hence
\begin{align*}
[b]_{t^d}/[b]_t &= \sum_{\x\in B_0''} t^{\lambda(\x)} - \sum_{\x\in B_0} t^{\lambda(\x)}\\
&= \sum_{\x\in X_0} t^{\lambda(\x)} - \sum_{\x\in X_1} t^{\lambda(\x)}.
\end{align*}
This completes the proof of part (a). To prove part (b), observe that the terms $t^{\lambda(\x)}$ in $[b]_{t^d}/[b]_t$ with $\lambda(\x)\equiv r$ mod $d$ come from the points $\x$ in the row containing label $r$, i.e., in the row containing label $\delta_r b$. Let $r\in\{0,\ldots,d-1\}$ and define $\gamma_r:=(\delta_r b-r)/d\in\NN$. This is the ``width'' of the triangle $B_0$ in the row containing label $r$, as in Figure \ref{fig:two_triangles}. If $\delta_r<\delta_1$ then the row containing $r$ intersects the rectangle $X_0$ at the sequence of points with labels $\delta_rb, \delta_rb+d,\ldots, \delta_rb+(\beta_1-1)d$. The sum of $t^{\lambda(\x)}$ over these points is
\begin{equation*}
\sum_{k=0}^{\beta_1-1} t^{\delta_r b+kd} = t^{\delta_r b} [\beta_1]_{t^d} = t^r (t^d)^\gamma_r [\beta_1]_{t^d}.
\end{equation*}
Hence
\begin{equation*}
\sum_{\x\in X_0} t^{\lambda(\x)} = \sum_{\substack{0\leq r\leq d-1 \\ \delta_r<\delta_1}} t^r(t^d)^{\gamma_r}[\beta_1]_{t^d}.
\end{equation*}
If $\delta_r\geq \delta_1$ then the row containing $r$ intersects the rectangle $X_1$ at the sequence of points with labels $c, c+d, \ldots, c+(b-\beta_1-1)d$, where $c=r+(\gamma_r-b+\beta_1)d$. (In Figure \ref{fig:two_triangles} we have used the letter $s$ for this case instead of $r$.) The sum of $t^{\lambda(\x)}$ over these points is
\begin{equation*}
\sum_{k=0}^{b-\beta_1-1} t^{c+kd} = t^c [b-\beta_1]_{t^d} = t^r (t^d)^{\gamma_r-b+\beta_1} [b-\beta_1]_{t^d}.
\end{equation*}
Hence
\begin{equation*}
\sum_{\x\in X_1} t^{\lambda(\x)} = \sum_{\substack{0\leq r\leq d-1 \\ \delta_r\geq \delta_1}}  t^r (t^d)^{\gamma_r-b+\beta_1} [b-\beta_1]_{t^d}.
\end{equation*}
\end{proof}

Finally, we apply Theorem \ref{thm:interesting} to obtain a more precise formula for the integers $\omega^k\sqbinom{n}{k}_\omega^{(b)}$ when $\omega^{n+1}=1$. In particular, this formula determines the sign and vanishing of these integers.

\begin{corollary}\label{cor:interesting}
Let $\omega$ be a primitive $d$th root of unity with $d|(n+1)$ and let $\gcd(b,d)=1$. Then for any $r\in\{0,\ldots,d-1\}$ we have
\begin{equation*}
\sum_{k\geq 0} \omega^{r+dk}\sqbinom{n}{r+kd}^{(b)}_\omega t^{r+dk}= t^r g_r(t^d),
\end{equation*}
where
\begin{equation*}
g_r(t) = \begin{cases}
t^{\gamma_r}[\beta_1]_t [b]_t^{(n+1)/d-1}& \delta_r<\delta_1,\\
-t^{\gamma_r-b+\beta_1} [b-\beta_1]_t[b]_t^{(n+1)/d-1} & \delta_r\geq \delta_1.
\end{cases}
\end{equation*}
We observe that each $g_r(t)$ is (plus or minus) a unimodal polynomial since products of $t$-integers are unimodal; see \cite[Theorem 3.9]{andrews}.
\end{corollary}

\begin{proof}
Recall from Theorem \ref{thm:main}c that we have
\begin{equation*}
\sum_{k= 0}^{(b-1)n} \omega^k\sqbinom{n}{k}^{(b)}_\omega t^k = \frac{[b]_{t^d}}{[b]_t} [b]_{t^d}^{(n+1)/d-1}.
\end{equation*}
Now use Theorem \ref{thm:interesting}b.
\end{proof}

To end this section, we complete our example with $(b,d)=(7,5)$. Recall that the relevant permutations are $(\beta_0,\ldots,\beta_6)=(0,3,6,2,5,1,4)$ and $(\delta_0,\ldots,\delta_4)=(0,3,1,4,2)$, hence $\beta_1=3$ and $b-\beta_1=4$. To compute the $5$-multisections of $[7]_{t^5}/[7]_t$ we note that the numbers $\gamma_r=(\delta_r b-r)/d$ are $(\gamma_0,\ldots,\gamma_4)=(0,4,1,5,2)$, which are the lengths of the rows in the fundamental triangle $B_0$ (including the empty bottom row), permuted according to congruence class mod $d=5$. Thus we have
\begin{align*}
[7]_{t^5} / [7]_t &= \sum_{\substack{0\leq r\leq 4 \\ \delta_r<3}} t^r(t^5)^{\gamma_r}[3]_{t^5} - \sum_{\substack{0\leq r\leq 4 \\ \delta_r\geq 3}}  t^r (t^5)^{\gamma_r-4} [4]_{t^5}  \\
&= t^0[3]_{t^5} -t^1[4]_{t^5} +t^2(t^5)^1[3]_{t^5}-t^3(t^5)^1[4]_{t^5}+t^4(t^5)^2[3]_{t^5}.
\end{align*}
Finally, if $\omega$ is a $d=5$th root of unity and $5|(n+1)$, then we have
\begin{equation*}
\sum_{k= 0}^{6n} \omega^k\sqbinom{n}{k}^{(7)}_\omega t^k = \frac{[7]_{t^5}}{[7]_t} [7]_{t^5}^{(n+1)/5-1},
\end{equation*}
and hence
\begin{equation*}
\sum_{k\geq 0} \omega^{r+5k}\sqbinom{n}{r+5k}^{(7)}_\omega t^{r+5k}= 
\begin{cases}
[3]_{t^5} [7]_{t^5}^{(n+1)/5-1} & r=0,\\
-t^1[4]_{t^5} [7]_{t^5}^{(n+1)/5-1} & r=1,\\
t^2(t^5)^1[3]_{t^5} [7]_{t^5}^{(n+1)/5-1} & r=2,\\
-t^3(t^5)^1[4]_{t^5} [7]_{t^5}^{(n+1)/5-1} & r=3,\\
 t^4(t^5)^2[3]_{t^5} [7]_{t^5}^{(n+1)/5-1} & r=4.
\end{cases}
\end{equation*}
In particular, we have $\sqbinom{n}{k}^{(7)}_\omega\geq 0$ when $\text{($k$ mod $5$)}\in \{0,2,4\}$ and $\sqbinom{n}{k}^{(7)}_\omega\leq  0$ when  $\text{($k$ mod $5$)}\in \{1,3\}$. The criterion for $\sqbinom{n}{k}^{(7)}_\omega= 0$ is a bit complicated to state.
 
\section{Other Symmetric Polynomials}\label{sec:other}

In this paper we have studied the principal specialization of the $b$-bounded symmetric polynomials $h_k^{(b)}$ at roots of unity. It may be interesting to compare these specializations to the expansion of $h_k^{(b)}$ in various bases for the ring of symmetric polynomials. Here we follow the standard notation from Macdonald \cite{macdonald}.

An {\em integer partition} is a sequence $\lambda=(\lambda_1,\lambda_2,\ldots,)\in\NN^\infty$ satisfying $\lambda_1\geq \lambda_2\geq \cdots \geq 0$, where only finitely many entries are non-zero. We define the {\em weight of $\lambda$} (also called {\em size}) by $|\lambda|=\sum_i \lambda_i$. If $|\lambda|=k$ then we say that $\lambda$ is a {\em partition of $k$} (another notation is $\lambda\vdash k$). We let $l(\lambda)$ denote the maximum $\ell$ such that $\lambda_\ell\neq 0$, which is called the {\em length of $\lambda$}. We write $m_i(\lambda)=\#\{j:\lambda_j=i\}$ (or just $m_i$ when $\lambda$ is understood) for the number of parts of $\lambda$ equal to $i$. We can also express this colloquially as $\lambda=(1^{m_1}2^{m_2}\cdots)$. Note that $|\lambda|=\sum_i im_i$. The following quantity also shows up frequently:\footnote{The integer $z_\lambda$ has nothing to do with the variables $z_i$.}
\begin{equation*}
z_\lambda = \prod_{i\geq 1} i^{m_i}\cdot m_i!.
\end{equation*}
Its significance is the fact that for $|\lambda|=n$ the number of permutations of $\{1,2,\ldots,n\}$ with cycle type $\lambda$ is $n!z_\lambda^{-1}$. For any integer partition $\lambda$ we define the corresponding elementary and complete symmetric polynomials
\begin{align*}
e_\lambda(z_1,\ldots,z_n) &:= \prod_i e_{\lambda_i}(z_1,\ldots,z_n),\\
h_\lambda(z_1,\ldots,z_n) &:= \prod_i h_{\lambda_i}(z_1,\ldots,z_n),
\end{align*}
with the convention that $e_0=h_0=1$. By analogy we define the $b$-bounded polynomials
\begin{equation*}
h_\lambda^{(b)}(z_1,\ldots,z_n) :=\prod_i h_{\lambda_i}^{(b)}(z_1,\ldots,z_n),
\end{equation*}
so that $h_\lambda^{(2)}=e_\lambda$ and $h_\lambda^{(\infty)}=h_\lambda$. Furthermore, we define the {\em power sum symmetric polynomials} $p_\lambda(z_1,\ldots,z_n):=\prod_i p_{\lambda_i}(z_1,\ldots,z_n)$, where
\begin{equation*}
p_k(z_1,\ldots,z_n)=z_1^k+\cdots+z_n^k,
\end{equation*}
and $p_0=1$. The Fundamental Theorem of Symmetric Polynomials says that the set $\{e_\lambda\}$ is a $\ZZ$-linear basis for the ring $\Lambda_n=\ZZ[z_1,\ldots,z_n]^{S_n}$ of symmetric polynomials with integer coefficients. It is also true that $\{h_\lambda\}$ is a $\ZZ$-linear basis for $\Lambda_n$, while the set $\{p_\lambda\}$ is a $\QQ$-linear basis for $\Lambda_n\otimes \QQ$. The expansions of $e_k$ and $h_k$ in terms of $p_\lambda$ are called ``Newton's identities'' \cite[Equation 2.14$'$]{macdonald}:
\begin{align*}
h_k(z_1,\ldots,z_n) &=\sum_{|\lambda|=k} z_\lambda^{-1} p_\lambda(z_1,\ldots,z_n),\\
e_k(z_1,\ldots,z_n) &=\sum_{|\lambda|=k} z_\lambda^{-1}(-1)^{|\lambda|-l(\lambda)}p_\lambda(z_1,\ldots,z_n).
\end{align*}
The generalization of these formulas to $b$-bounded symmetric polynomials follows from Proposition 3.15 and Remark 3.16(2) in Doty and Walker \cite{dw}, which they attribute to Macdonald. We think it is interesting to reproduce the proof here. Note that Doty and Walker did not consider the specialization $z_i=1$. In Corollary \ref{cor:hbpow} we will consider the specializations $z_i=\omega^{i-1}$ and $z_i=\omega^i$ for roots of unity $\omega$.

\begin{proposition}\label{prop:hbpow}
Let $b,n,k\in\NN$ with $b\geq 2$. For an integer partition $\lambda$ let $l_b(\lambda)$ denote the number of parts of $\lambda$ that are divisible by $b$. Then we have
\begin{equation*}
h_k^{(b)}(z_1,\ldots,z_n) = \sum_{|\lambda|=k} z_\lambda^{-1}(1-b)^{l_b(\lambda)} p_\lambda(z_1,\ldots,z_n),
\end{equation*}
and substituting $z_i=1$ for all $i$ gives
\begin{equation*}
\binom{n}{k}^{(b)} = \sum_{|\lambda|=k} z_\lambda^{-1}(1-b)^{l_b(\lambda)}n^{l(\lambda)}.
\end{equation*}
\end{proposition}

For example, we consider the case $(n,k)=(3,3)$ for various $b$. The partitions of $k=3$ are $\lambda\in\{(3),(2,1),(1,1,1)\}$, with $z_{(3)}=3$, $z_{(2,1)}=2$ and $z_{(1,1,1)}=6$. For any $b>3$  we have $l_b(\lambda)=0$ for all $|\lambda|=3$, hence
\begin{equation*}
\binom{3-1+3}{3}=\binom{3}{3}^{(b)}=\frac{(3)^1(1-b)^0}{3}+\frac{(3)^2(1-b)^0}{2}+\frac{(3)^3(1-b)^0}{6}=10.
\end{equation*}
When $b=3$ we have $l_3(3)=1$, $l_3(2,1)=0$ and $l_3(1,1,1)=0$, hence
\begin{equation*}
\binom{3}{3}^{(3)}=\frac{(3)^1(-2)^1}{3}+\frac{(3)^2(-2)^0}{2}+\frac{(3)^3(-2)^0}{6}=7.
\end{equation*}
When $b=2$ we have $l_2(3)=0$, $l_2(2,1)=1$ and $l_2(1,1,1)=0$, hence
\begin{equation*}
\binom{3}{3}=\binom{3}{3}^{(2)}=\frac{(3)^1(-1)^0}{3}+\frac{(3)^2(-1)^1}{2}+\frac{(3)^3(-1)^0}{6}=1.
\end{equation*}
When $b>k$ we note that $l_b(\lambda)=0$ for all $|\lambda|=k$ and when $b=2$ one can check that $|\lambda|-l(\lambda)$ and $l_2(\lambda)$ have the same parity, so these cases reduce to the classical Newton identities.

The key idea in the proof is the introduction of the polynomials $p_k^{(b)}$. Doty and Walker called these {\em modular power sums} and they defined them in a slightly different way. Walker \cite{walker} says that the formula relating $p_k^{(b)}$ and $p_k$ is a ``surprising observation'' of Macdonald. 

\begin{proof}
Recall the generating function
\begin{equation*}
H^{(b)}(t;z_1,\ldots,z_n)=\sum_{k\geq 0} h_k^{(b)}(z_1,\ldots,z_n) t^k = \prod_{i=1}^{n} \frac{1-z_i^b t^b}{1-z_i t}.
\end{equation*}
Taking logarithms gives
\begin{align*}
\log H^{(b)}(t;z_1,\ldots,z_n) &= \sum_{i=1}^n \log(1-z_i^b t^b) - \sum_{i=1}^n \log(1-z_i t)\\
&= -\sum_{i=1}^n \sum_{k\geq 1} \frac{1}{k} z_i^{kb} t^{kb} + \sum_{i=1}^n \sum_{k\geq 1} \frac{1}{k} z_i^k t^k\\
&= -\sum_{k\geq 1} \frac{1}{k} p_{kb}(z_1,\ldots,z_n) t^{kb} + \sum_{k\geq 1}\frac{1}{k} p_k(z_1,\ldots,z_n)t^k.
\end{align*}
It is convenient to write this last expression as
\begin{equation*}
\sum_{k\geq 1} \frac{1}{k} p_k^{(b)}(z_1,\ldots,z_n) t^k,
\end{equation*}
where we define
\begin{equation*}
p_k^{(b)}(z_1,\ldots,z_n) = \begin{cases}
(1-b)p_k(z_1,\ldots,z_n) & b|k,\\
p_k(z_1,\ldots,z_n) & b\nmid k.\\
\end{cases}
\end{equation*}
For any integer partition $\lambda$ we also write
\begin{equation*}
p_\lambda^{(b)}(z_1,\ldots,z_n)=\prod_i p_{\lambda_i}^{(b)}(z_1,\ldots,z_n)=(1-b)^{l_b(\lambda)}p_\lambda(z_1,\ldots,z_n).
\end{equation*}
Then taking exponentials gives\footnote{We omit variables to save space.}
\begin{align*}
H^{(b)}(t;z_1,\ldots,z_n) &= \prod_{i\geq 1} \exp\left(\frac{p_i^{(b)}t^i}{i}\right)\\
&= \prod_{i\geq 1} \sum_{m_i\geq 0} \frac{1}{m_i!} \left(\frac{p_i^{(b)}t^i}{i}\right)^{m_i}\\
&= \prod_{i\geq 1} \sum_{m_i \geq 0} \frac{1}{i^{m_i}\cdot m_i!}(p_i^{(b)})^{m_i} t^{im_i}\\
&= \sum_{k\geq 0} \sum_{|\lambda|=k} z_\lambda^{-1} p_\lambda^{(b)} t^k\\
&= \sum_{k\geq 0} \sum_{|\lambda|=k}  z_\lambda^{-1}(1-b)^{l_b(\lambda)} p_\lambda t^k.
\end{align*}
\end{proof}

Let us see what happens to Proposition \ref{prop:hbpow} under specialization at roots of unity. The specialization of the power sums is easy, so we omit the proof.

\begin{lemma}\label{lem:pow}
Let $\omega$ be a primitive $d$th root of unity. For an integer partition $\lambda$ recall that $l_d(\lambda)$ is the number of parts of $\lambda$ divisible by $d$.
\begin{enumerate}
\item If $d|n$ then we have
\begin{equation*}
p_k(\omega,\omega^2,\ldots,\omega^n)=p_k(1,\omega,\ldots,\omega^{n-1}) = \begin{cases} n & d|k,\\ 0 & d\nmid k,\end{cases}
\end{equation*}
and hence
\begin{equation*}
p_\lambda(\omega,\omega^2,\ldots,\omega^n)=p_\lambda(1,\omega,\ldots,\omega^{n-1}) = \begin{cases} n^{l(\lambda)} & d|\lambda_i \text{ for all $i$},\\ 0 & \text{else}.\end{cases}
\end{equation*}
\item If $d|(n-1)$ then we have
\begin{equation*}
p_k(1,\omega,\ldots,\omega^{n-1}) = \begin{cases} n & d|k,\\ 1 & d\nmid k,\end{cases}
\end{equation*}
and hence $p_\lambda(1,\omega,\ldots,\omega^{n-1})=n^{l_d(\lambda)}$.
\item If $d|(n+1)$ then we have
\begin{equation*}
p_k(\omega,\omega^2,\ldots,\omega^n) = \begin{cases} n & d|k,\\ -1 & d\nmid k,\end{cases}
\end{equation*}
and hence $p_\lambda(\omega,\omega^2,\ldots,\omega^n)=n^{l_d(\lambda)}(-1)^{l(\lambda)-l_d(\lambda)}=(-n)^{l_d(\lambda)}(-1)^{l(\lambda)}$.
\end{enumerate}
\end{lemma}

\begin{corollary}\label{cor:hbpow}
Let $\omega$ be a primitive $d$th root of unity.
\begin{enumerate}
\item If $d|n$ then combining Proposition \ref{prop:hbpow} with Lemma \ref{lem:pow}a gives
\begin{equation*}
\sqbinom{n}{k}_\omega^{(b)}= \sum_{\substack{ |\lambda|=k \\ d|\lambda_i \text{ for all } i}} z_\lambda^{-1} (1-b)^{l_b(\lambda)}n^{l(\lambda)}.
\end{equation*}
For each term in the sum we can write $\lambda=d\mu$ for some partition $\mu\vdash k/d$. Then using the facts $l(\lambda)=l(\mu)$ and $z_\lambda=d^{l(\mu)}z_\mu$ gives
\begin{equation*}
\sqbinom{n}{k}_\omega^{(b)}= \sum_{|\mu|=k/d} z_\mu^{-1} (1-b)^{l_b(d\mu)}(n/d)^{l(\mu)}.
\end{equation*}
In the special case $\gcd(b,d)=1$ we also have $l_b(d\mu)=l_b(\mu)$, and hence
\begin{equation*}
\sqbinom{n}{k}_\omega^{(b)}= \sum_{|\mu|=k/d} z_\mu^{-1} (1-b)^{l_b(\mu)}(n/d)^{l(\mu)}=\binom{n/d}{k/d}^{(b)},
\end{equation*}
which recovers the identity from Theorem \ref{thm:main}a.
\item If $d|(n-1)$ then we have
\begin{equation*}
\sqbinom{n}{k}_\omega^{(b)}= \sum_{|\lambda|=k} z_\lambda^{-1} (1-b)^{l_b(\lambda)}n^{l_d(\lambda)}.
\end{equation*}
\item If $d|(n+1)$ then we have
\begin{equation*}
\omega^k \sqbinom{n}{k}_\omega^{(b)} = \sum_{|\lambda|=k} z_\lambda^{-1}(1-b)^{l_b(\lambda)}(-1)^{l(\lambda)-l_d(\lambda)}  n^{l_d(\lambda)}.
\end{equation*}
\end{enumerate}
\end{corollary}

To end this paper we also consider the expansion of $h_k^{(b)}$ in the bases $e_\lambda$, $m_\lambda$ and $s_\lambda$. The {\em monomial symmetric polynomials} $m_\lambda(z_1,\ldots,z_n)$ are defined as follows. Let $\lambda$ be a partition of length $l(\lambda)=\ell\leq n$. We will write $\lambda=(\lambda_1,\ldots,\lambda_\ell,\lambda_{\ell+1},\ldots,\lambda_n)$ with $\lambda_{\ell+1}=\lambda_{\ell+2}=\cdots=\lambda_n=0$. Suppose that $\lambda=(0^{m_0}1^{m_1}2^{m_2}\cdots)$, i.e., that $\lambda$ has $m_i$ parts equal to $i$. In particular, we have $m_0=n-\ell$ and $\sum_i m_i=n$. Then we define\footnote{The use of the same letter $m$ for the unrelated $m_\lambda$ and $m_i$ is unfortunate.}
\begin{equation*}
m_{\lambda}(z_1,\ldots,z_n):= \frac{1}{\prod_{i=0}^n m_i!}\,\sum_{\pi\in S_n} z_1^{\lambda_{\pi(1)}}z_2^{\lambda_{\pi(2)}}\cdots z_n^{\lambda_{\pi(n)}} \in\ZZ[z_1,\ldots,z_n].
\end{equation*}
This can be more simply described as the sum over all distinct monomials in the variables $z_1,\ldots,z_n$ whose coefficients are a rearrangement of the list $\lambda_1,\ldots,\lambda_n$. For example, we have
\begin{align*}
m_{(2,2)}(z_1,z_2,z_3)=m_{(2,2,0)}(z_1,z_2,z_3)&=z_1^2z_2^2+z_1^2z_3^2+z_2^2z_3^2,\\
m_{(2,1)}(z_1,z_2,z_3)=m_{(2,1,0)}(z_1,z_2,z_3)&=z_1^2z_2+z_2^2z_1+z_1^2z_3+z_3^2z_1+z_2^2z_3+z_3^2z_2.
\end{align*}
The {\em Schur polynomials} $s_\lambda(z_1,\ldots,z_n)$ are most easily defined as a ratio of determinants of $n\times n$ matrices:
\begin{equation*}
s_\lambda(z_1,\ldots,z_n) = \det(z_i^{n-j+\lambda_j}) / \det(z_i^{n-j}) \in\ZZ[z_1,\ldots,z_n].
\end{equation*}
These polynomials have a rich theory that we will not discuss here. Given an integer partition $\lambda=(\lambda_1,\lambda_2,\ldots)$, the {\em conjugate partition} $\lambda'=(\lambda_1',\lambda_2',\ldots)$ is defined by $\lambda_i':=\#\{j: \lambda_j \geq i\}$, so that $m_i(\lambda)=\lambda_i'-\lambda_{i+1}'$. Note that $|\lambda|=|\lambda'|$. We have the following expansions for $h_k^{(b)}(z_1,\ldots,z_n)$ in terms of various bases for $\Lambda_n$, which follow easily from the theory of duality for symmetric functions (see Macdonald \cite[Chapter I.4]{macdonald}). None of these formulas is new, but we believe it is interesting to show that they all have the same proof. After the proof we will give references and discuss possible connections to Theorem \ref{thm:main}.

\begin{proposition}\label{prop:cauchy} Let $n,k,b\in\NN$ with $b\geq 2$ and let $\zeta$ be a primitive $b$th root of unity. Then the polynomial $h_k^{(b)}(z_1,\ldots,z_n)$ has the following expressions:
\begin{enumerate}
\item $\sum_{|\lambda|=k} (-1)^{-l(\lambda)} z_\lambda^{-1} p_\lambda(\zeta,\ldots,\zeta^{b-1}) p_\lambda(z_1,\ldots,z_n)$,
\item $(-1)^k \sum_{|\lambda|=k} e_\lambda(\zeta,\ldots,\zeta^{b-1}) m_\lambda(z_1,\ldots,z_n)$,
\item $(-1)^k \sum_{|\lambda|=k} m_\lambda(\zeta,\ldots,\zeta^{b-1}) e_\lambda(z_1,\ldots,z_n)$,
\item $(-1)^k \sum_{|\lambda|=k} s_{\lambda'}(\zeta,\ldots,\zeta^{b-1}) s_\lambda(z_1,\ldots,z_n)$.
\end{enumerate}
\end{proposition}

\begin{proof}

For any variables $x_1,x_2,\ldots,y_1,y_2,\ldots$, the {\em dual Cauchy kernel} has the following expansions, where each sum is over all integer partitions $\lambda$ \cite[Equations 4.1,$'$, 4.2$'$, 4.3$'$]{macdonald}:
\begin{align*}
\prod_{i=1}^\infty \prod_{j=1}^{\infty}(1+x_iy_j) &= \sum_{\lambda} (-1)^{|\lambda|-l(\lambda)} z_\lambda^{-1}p_\lambda(x_1,x_2,\ldots)p_\lambda(y_1,y_2,\ldots)\\
&= \sum_{\lambda} e_\lambda(x_1,x_2,\ldots)m_\lambda(y_1,y_2,\ldots)\\
&= \sum_{\lambda} m_\lambda(x_1,x_2,\ldots)e_\lambda(y_1,y_2,\ldots)\\
&= \sum_{\lambda} s_{\lambda'}(x_1,x_2,\ldots)s_\lambda(y_1,y_2,\ldots),
\end{align*}
Now let $\zeta$ be a primitive $b$th root of unity. We can express the generating function for $h_k^{(b)}$ as
\begin{align*}
H^{(b)}(t;z_1,\ldots,z_n) &= \sum_k h_k^{(b)}(z_1,\ldots,z_n) t^k\\
 &= \prod_{j=1}^n (1+z_jt+\cdots+(z_jt)^{b-1}) \\
&= \prod_{i=1}^{b-1} \prod_{j=1}^{n} (1-\zeta^i z_jt).
\end{align*}
We can obtain this generating function by substituting
\begin{align*}
(x_1,x_2,\ldots) &= (\zeta,\ldots,\zeta^{b-1},0,0,\ldots),\\
(y_1,y_2,\ldots) &= (-z_1t,-z_2t,\ldots,-z_nt,0,0,\ldots),
\end{align*}
into the dual Cauchy kernel. Then to compute the coefficient of $t^k$ we use the fact that
\begin{equation*}
f_\lambda(-z_1t,-z_2t,\ldots,-z_nt) = (-t)^{|\lambda|} f_\lambda(z_1,\ldots,z_n)
\end{equation*}
for any homogeneous polynomial $f_\lambda$ of degree $|\lambda|$.
\end{proof}

Proposition \ref{prop:cauchy}a is equivalent to Proposition \ref{prop:hbpow} because Lemma \ref{lem:pow}c says
\begin{equation*}
p_\lambda(\zeta,\ldots,\zeta^{b-1})=(b-1)^{l_b(\lambda)}(-1)^{l(\lambda)-l_b(\lambda)}=(1-b)^{l_b(\lambda)}(-1)^{l(\lambda)}.
\end{equation*}
Recall from Theorem \ref{thm:main}c that
\begin{equation*}
e_k(\omega,\omega^2,\ldots,\omega^{b-1})=\begin{cases}
(-1)^k & k<b,\\
0 & k\geq b.
\end{cases}
\end{equation*}
It follows from this that
\begin{equation*}
e_\lambda(\zeta,\ldots,\zeta^{b-1}) = \begin{cases}
(-1)^{|\lambda|} & \lambda_i<b \text{ for all $i$},\\
0 & \text{else}.
\end{cases}
\end{equation*}
Thus Proposition \ref{prop:cauchy}b just says that $h_k^{(b)}(z_1,\ldots,z_n) = \sum_\lambda m_\lambda(z_1,\ldots,z_n)$, where the sum is over partitions $\lambda$ with $|\lambda|=k$ in which each part is less than $b$, which is obviously true.

The last two parts are more interesting. Proposition \ref{prop:cauchy}c appears in Doty and Walker \cite[Remark 3.10(3)]{dw}. We do not know a general formula for the integers $m_\lambda(\zeta,\ldots,\zeta^{b-1})$. Yamaguchi et al.~\cite{yamaguchi} give formulas for some special cases. If $\omega$ is a primitive $d$th root of unity then the specializations of Proposition \ref{prop:cauchy}c at $z_i\mapsto \omega^i$ and $z_i\mapsto \omega^{i-1}$ should involve an interesting mix of $b$th and $d$th roots of unity. 

The coefficients $s_{\lambda'}(\zeta,\ldots,\zeta^{b-1})$ in Proposition \ref{prop:cauchy}d were studied by Grinberg. His main result \cite[Corollary 2.9]{grinberg} says that
\begin{equation*}
h_k^{(b)}(z_1,\ldots,z_n) = \sum_{|\lambda|=k} \pet_b(\lambda) s_\lambda(z_1,\ldots,z_n),
\end{equation*}
where $\pet_b(\lambda)\in\{0,-1,1\}$ is the determinant of a certain $b\times b$ {\em Petrie matrix} that has entries in $\{0,1\}$ and is determined by $\lambda$. Cheng et al.~\cite[Theorem 1.3]{chen} gave an expression for $\pet_b(\lambda)$ in terms of the $b$-core and $b$-quotient of the partition $\lambda$. When $\omega$ is a primitive $d$th root of unity with $d|n$, Reiner, Stanton and White \cite[Proposition 4.3]{rsw} gave an expression for $s_\lambda(1,\omega,\ldots,\omega^{n-1})$ in terms of the $d$-core and $d$-quotient of $\lambda$. It would be interesting to see how these two formulas interact, particularly in the case $\gcd(b,d)=1$. Perhaps there is a nice connection to the theory of simultaneous core partitions.

\section*{Acknowledgements} The author thanks Seamus Albion, Darij Grinberg, Christian Krattenthaler, Daniel Provder,  Vic Reiner, Brendon Rhoades, Bruce Sagan and Nathan Williams for helpful conversations, and Heather Armstrong for help with \LaTeX.

\end{document}